\documentclass[oneside,a4paper,english]{article}
\usepackage[utf8]{inputenc}
\usepackage[T1]{fontenc}
\usepackage{babel}
\usepackage{amssymb}
\usepackage{graphicx}
\usepackage{mathtools}
\usepackage{amsthm}
\usepackage{xfrac}
\usepackage{amsmath}
\usepackage{enumitem}
\usepackage{pgfplots}
\usepackage{pgfkeys}
\usepackage{array}
\usepackage{booktabs}
\usepackage{caption}
\usepackage{subfig}
\usepackage[left=2.4cm,right=2.4cm]{geometry}
\pgfplotsset{compat=1.15}
\usepackage{mathrsfs}
\usepackage{float}
\usetikzlibrary{arrows}
\usepackage{tcolorbox}
\usepackage{xcolor}
\usepackage{authblk}

\theoremstyle{definition} 
\newtheorem{definition}{Definition}[section]
\newtheorem{example}[definition]{Example}
\newtheorem{remark}[definition]{Remark}
\newtheorem{notation}[definition]{Notation}

\newtheorem{theorem}[definition]{Theorem}
\newtheorem{corollary}[definition]{Corollary}
\newtheorem{lemma}[definition]{Lemma}
\newtheorem{proposition}[definition]{Proposition}
\newtheorem*{maintheorem}{Main Theorem}
\newtheorem*{theorem*}{Theorem}

\newcommand{\F}{\mathbb F}
\newcommand{\Fq}{\mathbb F_q}
\newcommand{\Fm}{\mathbb F_{q^m}}

\DeclareMathOperator{\rk}{rk}
\DeclareMathOperator{\wt}{wt}
\DeclareMathOperator{\rank}{rank}
\DeclareMathOperator{\drk}{d_{rk}}
\DeclareMathOperator{\Trm}{Tr_{q^m/q}}

\usepackage[colorlinks=true, linkcolor=blue, citecolor=red, urlcolor=blue]{hyperref}

\title{Covering radius of rank-metric codes \\ via covering lifts and clubs}

\author[1]{Giuseppe Marino}
\author[1]{Alessandro Neri}
\author[1]{Alessia Vallefuoco}

\affil[1]{Department of Mathematics and Applications ``R. Caccioppoli'', University of Naples Federico II}

\affil[ ]{{\small
\texttt{giuseppe.marino@unina.it},
\texttt{alessandro.neri@unina.it},
\texttt{alessia.vallefuoco@unina.it}}
}

\date{}

\begin{document}
\maketitle

\begin{abstract}
We introduce a generator-matrix-based geometric approach to the covering radius of $\Fm$-linear rank-metric codes. Starting from the $q$-system associated with the code, we define the notion of covering lift and show that the covering radius can be recovered from the hyperplane weights of such lifts. This yields an exact criterion, in terms of the length and effective length of the code, for the covering radius to attain its largest possible value, together with the upper bound $\rho(\mathcal{C})\le \min\{m,n\}-1$ in all remaining cases. We then specialize our approach to $1$-dimensional codes, for which the effective length coincides with the minimum rank distance. In this setting, codes attaining $\rho(\mathcal{C})= \min\{m,n\}-1$ are related to covering lifts that are clubs. More generally, we characterize this case through linear maps into suitable quotient spaces, obtaining equivalent interpretations in terms of generalised evasive subspaces and auxiliary matrix rank-metric codes. These results determine the covering radius for several boundary values of the minimum distance and for all $1$-dimensional codes with extension degree $m\in\{3,4,5\}$. For $m=6$ and every $q$ we provide constructions and, for $q\in\{2,3\}$, computational results showing that $1$-dimensional rank-metric codes with the same length and minimum distance can have different covering radii.
\end{abstract}



\section{Introduction}
\paragraph{Background.} The covering radius is a classical parameter in coding theory. Given a code $\mathcal C$ in a metric space, its covering radius measures the largest possible distance between an element of the ambient space and the code. In this paper, we study the covering radius in the rank-metric setting. Let $\Fm$ be an extension of degree $m$ of $\Fq$. For $v=(v_1,\ldots,v_n)\in\Fm^n$, its rank weight is
$$\rk(v)=\dim_{\Fq}(\langle v_1,\ldots,v_n\rangle_{\Fq}),$$
and the associated rank distance is defined as $\drk(u,v)=\rk(u-v)$, for $u,v\in\Fm^n.$ An $\Fm$-linear rank-metric code $\mathcal C$ of length $n$ and dimension $k$ is an $\Fm$-subspace of $\Fm^n$ of dimension $k$. If $d$ denotes its minimum rank distance, we refer to such a code as an $[n,k,d]_{q^m/q}$ code. In this setting, the covering radius of $\mathcal C$ is $$\rho(\mathcal C)=\max_{x\in\Fm^n} d_{\rm rk}(x,\mathcal C)=\max_{x\in\Fm^n}\min_{c\in\mathcal C}\rk(x-c).$$
Equivalently, it is the smallest integer $\rho$ such that the balls of radius $\rho$ centered at the codewords cover the whole ambient space. Hence, while the minimum rank distance is an intrinsic parameter of the code, determined by the ranks of its nonzero codewords, the covering radius is an ambient parameter: it measures how far vectors of the whole space $\Fm^n$ can lie from $\mathcal C$. Its determination requires controlling simultaneously all the cosets of the code, and is in general a considerably difficult problem.
Since every vector of $\Fm^n$ has rank at most $m$, and every coset of $\mathcal C$ contains a vector with at most $n-k$ nonzero entries, one always has $\rho(\mathcal C)\le\min\{m,n-k\}$; see Section~\ref{sec:comparison}.
The study of covering radii links coding theory with covering problems, and has led to a rich interaction between coding theory, combinatorics and finite geometry, see e.g. \cite{CHLL1997}.

\medskip 
The covering radius has been extensively investigated in the Hamming-metric setting, see e.g. \cite{CHLL1997,CKMS1985}. From a geometric point of view, linear codes in the Hamming metric can be represented through projective systems, allowing several code parameters to be translated into incidence properties of finite projective spaces; see e.g. \cite{TV2007}. By this correspondence, the covering radius of a linear code admits a geometric counterpart in terms of saturating sets, geometric objects that were first introduced by Ughi in \cite{U1987}. This correspondence has led to a fruitful interaction between coding theory and finite geometry, and many constructions of
linear codes with small covering radius arise from geometric constructions of saturating sets, see e.g. \cite{D1995, DGMP2011, D2022}.

\medskip

The situation is rather different for the rank-metric setting. Although rank-metric codes have been widely studied, comparatively fewer results are known about their covering radius. A systematic study of the covering radius of rank-metric codes, with general upper and lower bounds, was carried out by Gadouleau and Yan in \cite{GY2008}, while further upper bounds, including the dual distance and external distance bounds, were established by Byrne and Ravagnani for matrix rank-metric codes \cite{BR2017}. More recently, Bonini, Borello and Byrne introduced rank-saturating systems, that are the $q$-analogues of saturating sets, and obtained a geometric interpretation of the rank-metric covering radius \cite{BBB2023}. This approach relies on the parity-check matrix of the code, in analogy with the classical correspondence between the covering radius of linear codes and saturating sets. Constructions of rank-saturating systems of the smallest possible rank can be found e.g. in \cite{BBM2024, BBB2023, Z2024}. A further geometric interpretation for matrix rank-metric codes whose minimum distance equals their number of rows was given in terms of generalised evasive subspaces in \cite{GRSZ2024}.

\paragraph{Our contributions.} In this paper, we develop a different geometric approach to the covering radius of $\Fm$-linear rank-metric codes, based on the generator matrix of the code rather than on its parity-check matrix.
 
Let $\mathcal C$ be an $[n,k]_{q^m/q}$ code with generator matrix $G$, and let $U\subseteq \Fm^k$ be the $\Fq$-span of the columns of $G$. The dimension $n'$ of $U$ is the \emph{effective length} of $\mathcal C$, and $\delta=n-n'$ is the number of coordinates which, up to equivalence, vanish on the whole code.
The starting point is an elementary observation (Lemma~\ref{Lemmacr}): a vector $\alpha$ is far from $\mathcal C$ exactly when, after adding $\alpha$ as a new row of $G$, all the codewords that really involve $\alpha$ have large rank. Following \cite{R2020}, the ranks of these codewords can be read as weights of hyperplanes, where the weight of a subspace is the dimension of its intersection with the $\Fq$-span $\widetilde U\subseteq\Fm^{k+1}$ of the columns of the extended matrix. This subspace projects onto $U$ and, as long as $\delta\le m$, it can be assumed to meet the point $P=\langle(0,\ldots,0,1)\rangle_{\Fm}$ in dimension $\delta$; we call such subspaces \emph{$\delta$-covering lifts} of $U$ (Definition~\ref{def:covering_lift}). In a covering lift, the hyperplanes through $P$ have weights determined by $U$ alone (Lemma~\ref{lem:cl_weights}), while the hyperplanes not through $P$ encode the distance of $\alpha$ from $\mathcal C$: the smaller their largest weight $\mu_P(\widetilde U)$, the farther $\alpha$ is from the code. This gives our first main result.

\begin{maintheorem}[see Theorem \ref{thm:hyp}]
Let $\mathcal C$ be an $[n,k]_{q^m/q}$ code with effective length $n'$, and assume that $\delta=n-n'\leq m$. Then
$$\rho(\mathcal C)=n-\min_{\widetilde U}\mu_P(\widetilde U),$$
where $\widetilde U$ runs over all $\delta$-covering lifts of $U$.
\end{maintheorem}

As a consequence, the effective length decides when the covering radius reaches its largest possible value $m$: this happens if and only if $\delta\ge m$, and otherwise $\rho(\mathcal C)\leq\min\{m,n\}-1$ (Proposition~\ref{prop:uppb}). For $0<\delta<m$ this improves on the dual distance bound of \cite{BR2017}, see Section~\ref{sec:comparison}.
 
\medskip
 
We then focus on $1$-dimensional codes, where the picture becomes much more concrete. Here the effective length coincides with the minimum distance $d$, and $U$ is simply a $d$-dimensional $\Fq$-subspace $V$ of $\Fm$ (Remark~\ref{rem:n'd}). Contrary to the Hamming metric, where all $[n,1,d]_q$ codes are equivalent, different choices of $V$ may give inequivalent codes (Lemma~\ref{lem:equivalence1d}). The covering lifts now lie in $\Fm^2$, whose hyperplanes are just points, so the Main Theorem asks for a lift in which all points other than $P$ have small weight. In the extreme case, all these points have weight at most one: such subspaces are known as \emph{clubs} \cite{FS2006, DBVdV2016, NPSZ2024}. More precisely, a $t$-club of rank $n$ is an $n$-dimensional $\Fq$-subspace of $\Fm^2$ having one point of weight $t$ and all other points of weight at most one.
 
\begin{theorem*}[Club characterization, see Corollary~\ref{cor:corrclub}]
Let $\mathcal C$ be an $[n,1,d]_{q^m/q}$ code with $m\geq n$. Then $\rho(\mathcal C)=n-1$ if and only if $V$ has an $(n-d)$-covering lift which is an $(n-d)$-club, with $P$ as the point of weight $n-d$.
\end{theorem*}

A similar sufficient condition holds when $m<n<m+d$ (Corollary~\ref{cor:n>m}). In this way, known constructions \cite{DBVdV2016, NPSZ2024} and nonexistence results \cite{DBVdV2022} for clubs translate directly into information on covering radii. Beyond clubs, we characterize when $\rho(\mathcal C)=\min\{m,n\}-1$ in terms of an $\Fq$-linear map $f:V\to\Fm$ and a $\delta$-dimensional $\Fq$-subspace $W\subseteq\Fm$, which together describe a suitable covering lift (Proposition~\ref{prop:caractd}). This condition can be read in terms of generalised evasive subspaces in the sense of \cite{GRSZ2024} (Remark~\ref{rem:d:evasive}), and also in terms of auxiliary matrix codes: the covering radius of $\mathcal C$ equals $\delta$ plus the largest covering radius among a family of $\Fq$-linear matrix codes $\mathcal A_{V,W}$ naturally attached to $V$ (Theorem~\ref{thm:assmat}).
 
\medskip
These geometric descriptions lead to several concrete determinations.
For $1$-dimensional rank-metric codes we completely determine the covering radius in the cases $d\in\{1,2,\min\{m,n\}-1,\min\{m,n\}\}$ (Theorem~\ref{thm:boundary}) and, combining this with the nonexistence of $2$-clubs of rank $5$, for every extension degree $m\in\{3,4,5\}$ (Proposition~\ref{prop:special}). In particular, every $[4,1]_{q^4/q}$ code has covering radius $3$, which recovers the main result proved in \cite{Z2024} via rank-saturating systems. For $m=6$, a known construction of $3$-clubs shows that $[6,1,3]_{q^6/q}$ codes for which $V$ is a scalar multiple of $\F_{q^3}$ have covering radius $5$ (Proposition~\ref{prop:consm6}), and MAGMA computations for $q\in\{2,3\}$ deal with the remaining open cases (Table~\ref{tab:m=6}). These computations show that $[6,1,3]_{q^6/q}$ codes can have covering radius $4$ or $5$: thus, the covering radius of a $1$-dimensional rank-metric code is not determined by its length and minimum distance, in sharp contrast with the Hamming metric.

\paragraph{Outline.} The paper is organized as follows. In Section \ref{sec:preliminaries}, we recall the basic notions concerning rank-metric codes, $q$-systems and $t$-clubs on the projective line. Section \ref{sec:mainthm} develops the general geometric approach, introduces the notion of covering lift and proves the main characterization and the resulting bounds in terms of the effective length. In Section \ref{sec:k1}, we specialize the results to $1$-dimensional codes, establish the connection with $t$-clubs, generalised evasive subspaces and auxiliary matrix codes, and apply these tools to boundary cases and small extension degrees. Finally, in Section \ref{sec:comparison}, we compare our results with the bounds previously obtained by Gadouleau and Yan \cite{GY2008}, by Byrne and Ravagnani \cite{BR2017} and by Gruica, Ravagnani, Sheekey and Zullo \cite{GRSZ2024}.

\section{Preliminaries} \label{sec:preliminaries}

\subsection{Rank-metric codes and $q$-systems}
We begin by introducing rank-metric codes, which provide the main setting for this paper. To this end, we fix some notation: throughout the paper, $q$ denotes a prime power, $m,n$ and $k$ denote positive integers and let $[n]=\{1,\ldots,n\}$. We consider the field extension $\Fm$ of degree $m$ of $\Fq$. Moreover, we need to endow $\Fm^n$ with a suitable distance.
\begin{definition}
    For $v=(v_1,\ldots,v_n)\in\Fm^n$, we can define its \textbf{rank weight} as 
    \begin{equation*}
        \rk(v)=\dim_{\Fq}(\langle v_1,\ldots,v_n\rangle_{\Fq}).
    \end{equation*}
\end{definition}
Let $\mathcal{B}$ be an $\Fq$-basis of $\Fm$. Then every element of $\Fm$ can be uniquely written as a linear combination of the elements in $\mathcal{B}$. In particular, for $v=(v_1,\ldots,v_n)\in\Fm^n$, every $i$-th coordinate can be identified with its coordinate vector with respect to $\mathcal{B}$, i.e. $v_i\equiv_{\mathcal{B}}(v_{1i},\ldots,v_{mi})\in\Fq^m$. Taking them as vector columns, we get an $m\times n$ matrix $\phi_{\mathcal B}(v)$ defining $v$. Equivalently, after fixing a basis $\mathcal B$, the vector $v\in\mathbb F_{q^m}^n$ can be represented by $\phi_{\mathcal B}(v)\in\mathbb F_q^{m\times n}$, and its rank weight coincides with the usual rank of this matrix
$$\rk(v)=\rank(\phi_{\mathcal B}(v)).$$
This rank weight induces a metric on $\Fm^n$, called the \textbf{rank metric}, which is defined by
\begin{equation*}
    \begin{array}{rccl}
\drk :\; &\Fm^n \times \Fm^n &\longrightarrow &\mathbb{N}_0,\\[2mm]
&(u,v) &\longmapsto &\rk(u-v).
\end{array}
\end{equation*}

\begin{definition}
A ($\Fm$-linear) \textbf{rank-metric code} is an $\Fm$-linear subspace $\mathcal{C}$ of $\Fm^n$. The positive integer $n$ is the \textbf{length} of the code, its \textbf{dimension} is the $\Fm$-dimension and its \textbf{minimum rank distance} is
\begin{equation*}
\mathrm{d}=\drk(\mathcal{C})=\min\{\rk(v):v\in\mathcal{C}\setminus\{0\}\}.
\end{equation*}
\end{definition}
If $\mathcal{C}$ has length $n$, dimension $k$ and minimum distance $d$, we will call it an $[n,k,d]_{q^m/q}$ code. If the minimum distance is not known or not relevant, we will say that $\mathcal{C}$ is an $[n,k]_{q^m/q}$ code.

As just defined, an $[n,k]_{q^m/q}$ code is an $\Fm$-subspace of dimension $k$. Hence, we can choose a basis of $\mathcal{C}$ and put its elements as rows of a $k\times n$ matrix. Such a matrix is called a generator matrix for $\mathcal{C}$. Equivalently, if $G$ is a generator matrix for $\mathcal{C}$, then $$\mathcal{C}=\{uG:u\in\Fm^k\}.$$
From now on, if $\mathcal{C}$ is an $[n,k]_{q^m/q}$ code, assume that $m,n>1$ and $0<k<n$.
Throughout the paper, we consider rank-metric codes only up to $\Fm$-linear rank-metric isometries of $\Fm^n$, see e.g. \cite{B2003,M2014,W1996}.
\begin{definition}
    Two $[n,k,d]_{q^m/q}$ rank-metric codes $\mathcal{C}$ and $\mathcal{C}'$ are \textbf{linearly equivalent} if there exists $A\in\mathrm{GL}(n,q)$ such that $\mathcal{C}'=\mathcal{C}A=\{vA:v\in\mathcal{C}\}$.
\end{definition}

Before proceeding, we introduce another parameter usually associated with a rank-metric code, called the effective length, which will be useful throughout the paper.
\begin{definition}
 Let $\mathcal{C}$ be an $[n,k]_{q^m/q}$ code, and let 
    $G=(G^1,\ldots,G^n)$ be a generator matrix for $\mathcal{C}$, where $G^i$ denotes the $i$-th column of $G$. 
    The \textbf{effective length} of $\mathcal{C}$ is defined as
   $$ n'=\dim_{\mathbb F_q}(\langle G^1,\ldots,G^n\rangle_{\mathbb F_q}).$$
    
     If $n'=n$, i.e. if the columns of $G$ are $\Fq$-independent, then $\mathcal{C}$ is \textbf{nondegenerate}.
\end{definition}

Before moving to the geometric counterpart of rank-metric codes, we introduce two standard operations that can be performed on a code: duality and puncturing.

\begin{definition}
Let $\mathcal{C}$ be an $[n,k]_{q^m/q}$ code. Its \textbf{dual code}, denoted by $\mathcal{C}^{\perp}$, is defined as 
\begin{equation*}
    \mathcal{C}^{\perp} = \{ u \in \Fm^n \mid  u\cdot c = 0 \;\text{ for every } c \in \mathcal{C} \},
\end{equation*}
where $\cdot$ denotes the standard inner product on $\Fm^n$.
\end{definition}

\begin{definition}
    Let $\mathcal{I}\subseteq[n]$ and consider the projection map
    $$\pi_{\mathcal{I}}:c\in\Fm^n\rightarrow c^{[n]\setminus\mathcal{I}}\in\Fm^{n-|\mathcal{I}|},$$
    where $c^{[n]\setminus\mathcal{I}}$ denotes the vector in $\Fm^{n-|\mathcal{I}|}$ obtained by deleting all the coordinates indexed by $\mathcal{I}$. Let $\mathcal{C}$ be an $[n,k]_{q^m/q}$ code, \textbf{the puncturing of $\mathcal{C}$} on $\mathcal{I}$ is 
    $$\Pi(\mathcal{C},\mathcal{I})=\{\pi_{\mathcal{I}}(c):c\in\mathcal
    C\}.$$
\end{definition}

The rest of this section will be devoted to the geometric interpretation of rank-metric codes in terms of $q$-systems. 

From now on, we use vector-space notation and in particular we identify $1$-dimensional $\Fm$-subspaces with the corresponding projective points. With a slight abuse of notation, if $H$ is an $\Fm$-subspace, we write $P\in H$ to mean that $P\subseteq H$.

As in the Hamming metric setting, where linear codes can be studied through their associated projective systems, also rank-metric codes admit a geometric counterpart, which is defined as follows.

\begin{definition}
    An $[n,k,d]_{q^m/q}$ system $U$ is an $\Fq$-subspace of $\Fm^k$ of $\Fq$-dimension $n$, not contained in any hyperplane (i.e. $\langle U\rangle_{\Fm}=\Fm^k$), with 
    \begin{equation*}
        d:=n-\max\{\dim_{\Fq}(U\cap H):H\subseteq\Fm^k, \dim_{\Fm}(H)=k-1\}.
    \end{equation*}
    Moreover, for an $\mathbb F_q$-subspace $X\subseteq \mathbb F_{q^m}^k$, its \textbf{weight} in $U$ is the quantity
    $$\wt_U(X):=\dim_{\mathbb F_q}(U\cap X).$$
\end{definition}

\begin{definition}
    Two $[n,k,d]_{q^m/q}$ systems $U$ and $U'$ are \textbf{linearly equivalent} if there exists an $\Fm$-isomorphism $\varphi\in\mathrm{GL}(k,q^m)$ such that $\varphi(U)=U'$.
\end{definition}

The geometric interpretation of rank-metric codes in terms of $q$-systems relies on the possibility of reading the rank of a codeword in terms of the weight of a hyperplane.

\begin{lemma}[{\cite{R2020}, \cite[Lemma 3.7]{ABNR2022}}]
    Let $\mathcal C$ be a nondegenerate $[n,k]_{q^m/q}$ code and let $G$ be a generator matrix. Then for every $v\in\Fm^k\setminus\{0\}$, 
    \begin{equation} \label{eq:qsyst}
    \rk(vG)=n-\wt_U(\langle v\rangle_{\Fm}^{\perp}),    
    \end{equation}
    where $U$ denotes the $[n,k]_{q^m/q}$-system given by the $\Fq$-span of the columns of $G$. 
\end{lemma}

In particular, given an $[n,k]_{q^m/q}$ nondegenerate code $\mathcal{C}$ and a generator matrix $G$ for $\mathcal{C}$, the $\Fq$-span of the columns of $G$, namely $G^1,\ldots,G^n$, forms an $[n,k]_{q^m/q}$ system $U$. Conversely, taken an $\Fq$-basis of an $[n,k]_{q^m/q}$ system $U$, the matrix $G$ containing the elements of the basis as columns, defines an $[n,k]_{q^m/q}$ nondegenerate code. Moreover, from \eqref{eq:qsyst}, if $d=\drk(\mathcal{C})$
\begin{align*}
    d=\min\{\rk(vG):v\in\Fm^k\setminus\{0\}\}
    &=\min\{n-\wt_U(\langle v\rangle_{\Fm}^{\perp}) : v\in\Fm^k\setminus\{0\} \}, \\
    &= n-\max\{\wt_U(\langle v\rangle_{\Fm}^{\perp}) : v\in\Fm^k\setminus\{0\}\},
\end{align*}
and since every hyperplane $H$ of $\Fm^k$ is of the form $\langle v\rangle_{\Fm}^\perp$ for some $v\in\Fm^k\setminus\{0\}$, the minimum rank distance of $\mathcal{C}$ coincides with the parameter $d$ associated with the $q$-system.

This correspondence is true up to equivalence and it is formally proved in \cite{R2020} and \cite{ABNR2022}.

\begin{theorem}[{\cite[Theorem 2]{R2020}}]\label{thm:correspondence}
    There exists a one-to-one correspondence between classes of $[n,k,d]_{q^m/q}$ nondegenerate rank-metric codes and classes of $[n,k,d]_{q^m/q}$ systems.
\end{theorem}

We conclude this preliminary section by recalling the notion of $t$-clubs, well-studied objects in finite geometry that will play a role in the next sections.

\subsection{$t$-clubs on the projective line}

The notion of clubs on the projective line was first introduced by Fancsali and Sziklai in \cite{FS2006}. However, the more general notion of $t$-club, in which the exceptional weight can be different from its rank minus 1, was introduced by De Boeck and Van de Voorde in \cite{DBVdV2016} in connection with translation KM-arcs. Then, more works provided constructions, classification results and nonexistence results for particular parameters \cite{DBVdV2016,DBVdV2022,NPSZ2024}, while a recent development includes clubs in higher-dimensional projective spaces \cite{MSZ2025}. Recently it has been also exploited the link between $t$-clubs and the free product of uniform $q$-matroids in \cite{ABF2025}. 

\begin{definition}
Let $U$ be an $\Fq$-subspace of $\Fm^2$. We say that $U$ is
\textbf{scattered of rank $n$} if it has $\Fq$-dimension $n$ and $\wt_U(P)\leq 1$ for every point $P\in\Fm^2$.
\end{definition}

\begin{definition}
    Let $t$ be a nonnegative integer and let $U$ be an $\Fq$-subspace of $\Fm^2$. We say that \textbf{$U$ is a $t$-club of rank $n$} if it has $\Fq$-dimension $n$ and there exists a point $P$ such that $\wt_U(P)=t$, while for every other point $Q\neq P$, $\wt_U(Q)\le1$. 
\end{definition}

\begin{remark}
In particular, for $t\in\{0,1\}$, the notion of $t$-club is closely related to that of scattered subspace. Indeed, for $t=1$, a $t$-club is simply a scattered subspace with $\wt_U(P)=1$, while for $t=0$, $U$ is again scattered, but the point $P$ is disjoint from $U$. We include these two cases in the definition of a club, as they will be useful later to simplify the notation.
\end{remark}

In this subsection we introduce a few results on $t$-clubs that will be needed later in the paper in order to recover lower bounds on the covering radius of rank-metric codes. The first result deals with the classification of $(m-2)$-clubs on the projective line provided in \cite{NPSZ2024}.

\begin{theorem}[{\cite[Theorems 3.2 and 3.6]{NPSZ2024}}]
\label{thm:clubs_cod2}
Let $m\geq3$, let $\beta\in\Fm\setminus\Fq$, and let $T\subseteq\Fm$
be an $(m-2)$-dimensional $\Fq$-subspace containing $1$.
Choose an $\Fq$-basis $t_1,\ldots,t_{m-2}$ of $T$. For $\gamma\in\Fm$,
let
$$U=\langle(1,1),(\beta,\gamma),(0,t_1),\ldots,(0,t_{m-2})\rangle_{\Fq}.$$
Then $\dim_{\Fq}(U)=m$ and
$\mathrm{wt}_U(\langle(0,1)\rangle_{\Fm})=m-2$.
Moreover, $U$ is an $(m-2)$-club if and only if
$\gamma\notin T+\beta T$, where $$T+\beta T:=\{v+\beta w:v,w\in T\}.$$
In particular, every $(m-2)$-club of rank $m$ is equivalent to one of this form.
\end{theorem}

We next recall some results that will be useful for the small parameter analysis in Section \ref{subsec:special_cases}. First, we recall a construction of clubs obtained from a scattered graph over a subfield, that will be helpful to construct $[6,1,3]_{q^6/q}$ codes with covering radius $5$.

\begin{proposition}[{\cite[Theorem 3.6]{DBVdV2016}}] \label{prop:3club}
Let $m=hr$, with $h,r>1$, and let $f:\F_{q^h}\rightarrow\F_{q^h}$ be an $\Fq$-linear map whose graph $\Gamma_f=\{(x,f(x)):x\in\F_{q^h}\}$ is scattered. For fixed $a,b\in\mathbb{F}_{q^h}$, with
$b\neq 0$, define
$$U_{a,b}=\left\{\left(f(x_0)-ax_0,\,bx_0+\sum_{i=1}^{r-1}x_i\omega^i\right):x_0,\ldots,x_{r-1}\in\mathbb{F}_{q^h}
\right\},$$ where $\{1,\omega,\ldots,\omega^{r-1}\}$
is an $\mathbb{F}_{q^h}$-basis of $\mathbb{F}_{q^m}$. Then
$U_{a,b}$ is an $t$-club of rank $m$, where
$$t=
\begin{cases}
h(r-1), & \text{if $f(x)-ax$ is invertible as an $\Fq$-linear endomorphism over $\mathbb{F}_{q^h}$,}\\
h(r-1)+1, & \text{otherwise.}
\end{cases}$$
\end{proposition}

We now recall a nonexistence result on clubs needed to compute the covering radius of $[5,1,3]_{q^5/q}$ codes.

\begin{theorem}[{\cite{DBVdV2022}}]
\label{thm:no2club}
For every prime power $q$, there is no $2$-club of rank $5$ in
$\mathbb F_{q^5}^2$.
\end{theorem}

As shown in Section \ref{sec:k1}, clubs are related to the study of the covering radius of $\Fm$-linear rank-metric codes, but not only, as highlighted by the following remark. 

\begin{remark}
Alfarano, Byrne and Fulcher identify clubs on the projective line with
representations of free products of two rank-one uniform $q$-matroids
\cite[Theorem~87]{ABF2025}. Their theorem ensures that if $U$ is a $t$-club of rank $n$, with $2\leq t<n$, then the $q$-matroid $M_U=(U,r_U)$, defined by the rank function $r_U(A)=\dim_{\Fm}(\langle A\rangle_{\Fm})$ where $A$ is an $\Fq$-subspace of $U$, satisfies
$$M_U\cong\mathcal U_{1,t}(q)\mathbin{\square}\mathcal U_{1,n-t}(q).$$
Here, $\mathcal U_{1,j}(q)$ denotes the uniform $q$-matroid of rank one
on a $j$-dimensional $\Fq$-space, and $\square$ denotes the free product.
Conversely, an $\Fm$-representation of this free product yields a
$t$-club of rank $n$. This means that our Corollary \ref{cor:corrclub}, for $\delta\ge 2$, can be also read in terms of free product of uniform $q$-matroids.
\end{remark}

\section{A geometric link with the covering radius} \label{sec:mainthm}

Starting from any distance, we can define the notion of covering radius. In our context, we focus on the covering radius for rank-metric codes.
\begin{notation}
    Let $x\in\Fm^n$, we will denote the distance between $x$ and $\mathcal C$ as
    $$\drk(x,\mathcal{C})=\min_{c\in\mathcal
    C}\drk(x,c).$$
\end{notation}

\begin{definition}
    The \textbf{covering radius} of an $[n,k,d]_{q^m/q}$ code $\mathcal{C}$ is
    \begin{equation*}
        \rho(\mathcal{C}):=\max_{x\in\Fm^n}\mathrm{d}_{\rk}(x,\mathcal{C}).
    \end{equation*}
\end{definition}

The notion of covering radius can be defined in any metric space. In particular,
in the Hamming metric setting, it has a geometric interpretation in terms
of saturating sets. An analogous geometric counterpart for rank-metric codes was introduced in \cite{BBB2023} using saturating systems, the $q$-analogues of saturating sets.

\begin{definition}
   An $[n,k]_{q^m/q}$ system $U$ is \textbf{rank-$\rho$-saturating} if, for every $v\in\Fm^k$ there exists a subspace $W\subset U$ of $\Fq$-dimension at most $\rho$ such that $v\in\langle W\rangle_{\Fm}$.
\end{definition}

In the Hamming metric, there exists a one-to-one correspondence between classes of codes with covering radius at most $\rho$ and classes of $\rho$-saturating sets. Similarly, there exists a correspondence between rank-metric codes and saturating systems.

\begin{theorem}[{\cite[Theorem 3.5]{BBB2023}}]\label{thmrankrho}
There exists a one-to-one correspondence between classes of $[n,n-k]_{q^m/q}$ codes whose dual is nondegenerate, with covering radius at most $\rho$, and classes of $[n,k]_{q^m/q}$ rank-$\rho$-saturating systems.
\end{theorem}

The study of the covering radius of a code is in general a difficult task. It is known that this problem can be approached using the correspondence with rank-$\rho$-saturating systems, as just observed. However, in some cases, it might be easier to use a different approach.  The stepping stone of our approach is based on the following elementary lemma.

\begin{lemma} \label{Lemmacr}
    Let $\mathcal{C}$ be an $[n,k]_{q^m/q}$ code and let $G$ be a generator matrix for $\mathcal{C}$. Let $t$ be a positive integer such that $t\leq\min\{m,n\}$. Then the following conditions are equivalent.
    \begin{enumerate}[label=(\arabic*)]
        \item \label{cond:rho} $\rho(\mathcal{C})\geq t$.
        \item \label{cond:equivrho}There exists $\alpha\in\Fm^n$ such that $\min_{c\in\mathcal{C}}\rk(\alpha-c)\geq t$.
        \item \label{cond:extended} There exists  $\alpha\in\Fm^n$ such that, if we take the code $\mathcal{C}_{\alpha}=\mathrm{rowsp}(G_{\alpha})$ with $G_{\alpha}=\begin{bmatrix}
            G \\
            \alpha
        \end{bmatrix}$, then
        \begin{equation*}
            \min\{\rk(u'G_{\alpha}):u'=(u_1,\ldots,u_{k+1})\in\Fm^{k+1}, u_{k+1}\neq0\}\geq t.
        \end{equation*}
    \end{enumerate}
\end{lemma}
\begin{proof}
    \underline{\ref{cond:rho} $\Longleftrightarrow$ \ref{cond:equivrho}}: It follows directly from the definition of covering radius.
    
    \underline{\ref{cond:equivrho} $\implies$ \ref{cond:extended}}: Let $\alpha\in\Fm^n$ such that $\min_{c\in\mathcal C}\rk(\alpha-c)\ge t$. Since $t\geq 1$, we necessarily have $\alpha\notin\mathcal C$. Otherwise, by taking $c=\alpha$, we would obtain $\rk(\alpha-c)=0$. Hence, $\alpha$ does not belong to the $\Fm$-row space of $G$, and $\rank(G_\alpha)=k+1$. Let $u'=(u_1,\ldots,u_k,u_{k+1})\in\Fm^{k+1}$ with $u_{k+1}\neq 0$, and set $u=(u_1,\ldots,u_k)\in\Fm^k$. Then $$u'G_\alpha=uG+u_{k+1}\alpha.$$
    Since $\mathcal{C}$ is $\Fm$-linear, then $-u_{k+1}^{-1}(uG)\in\mathcal{C}$. Moreover, multiplication by a nonzero element of $\Fm$ preserves the rank. Thus, 
$$\rk(u'G_{\alpha})=\rk\left({uG}+u_{k+1}\alpha\right)=\rk\left(\alpha-\left(-\frac{uG}{u_{k+1}}\right)\right)\ge t.$$
Since this holds for every $u'\in\Fm^{k+1}$ whose last coordinate is nonzero, then it is true for the minimum among all such vectors $u'$.

\underline{\ref{cond:extended} $\implies$ \ref{cond:equivrho}}: Conversely, consider $u_{k+1}=-1$, then
$$\rk(u'G_{\alpha})=\rk(uG-\alpha)=\rk(\alpha-uG)\ge t\quad\text{ for every } u=(u_1,\ldots,u_k)\in\Fm^k.$$
    Hence, we can conclude that there exists $\alpha\in\Fm^n$ such that $\min_{c\in\mathcal{C}}\rk(\alpha-c)\ge t$.
    
\end{proof}

The previous lemma allows us to introduce a different geometric approach for studying the covering radius of linear rank-metric codes. We will later provide evidence of the fact that this approach simplifies the discussion in some cases. Before doing this, let us introduce some notation.

\medskip
 Let $\mathcal{C}$ be an $[n,k]_{q^m/q}$ code with effective length $n'$, let $G$ be a generator matrix for $\mathcal{C}$ and denote by $G^i$, with $i\in[n]$ the $i$-th column of $G$. By definition of effective length, $\dim_{\Fq}(\langle G^1,\dots,G^n\rangle_{\Fq})=n'$, hence there exist $i_1,\ldots,i_{n'}\in[n]$ such that $\langle G^1,\dots,G^n\rangle_{\Fq}=\langle G^{i_1},\dots,G^{i_{n'}}\rangle_{\Fq}$. Therefore, there exists $A\in\operatorname{GL}(n,q)$ such that $GA=(\tilde{G}^1\cdots \tilde{G}^{n'}0\cdots0)$ with $\tilde{G}^1,\dots, \tilde{G}^{n'}$ linearly independent over $\Fq$. Thus, up to linear equivalence, every $[n,k]_{q^m/q}$ code $\mathcal{C}$ with effective length $n'$ admits a generator matrix of the form
 \begin{equation} \label{eq:genmatrix}
G=\begin{pmatrix}
        G^1\,\cdots \, G^{n'} \,0 \, \cdots\, 0
    \end{pmatrix},
 \end{equation}
    with $G^1,\ldots,G^{n'}$ linearly independent over $\Fq$. Since right multiplication by an element of $\operatorname{GL}(n,q)$ gives rank-metric isometries of $\Fm^n$, the covering radius is preserved by this operation, thus  $$\rho(\mathcal{C})=\rho(\mathcal{C}A) \quad\text{for every}\;\; A\in\operatorname{GL}(n,q).$$ In particular, if $n'<n$ then $\mathcal{C}$ is linearly equivalent to a code in $\Fm^{n'}\times\{0\}^{n-n'}$. Hence, it can be isometrically embedded in $\Fm^{n'}$ in the following way. Consider the code
    \begin{equation*}
        \tilde{\mathcal{C}}=\{(c_1,\ldots,c_{n'}):(c_1,\ldots,c_{n'},0,\ldots,0)\in\mathcal{C}A\}\subseteq \Fm^{n'}.
    \end{equation*}
 Since $$\rk(c_1,\ldots,c_{n'})=\rk(c_1,\ldots,c_{n'},0,\ldots,0) \quad \text{for every} \;\;(c_1,\ldots,c_{n'},0,\ldots,0)\in\mathcal{C}A,$$
$\tilde{\mathcal{C}}$ has the same dimension and minimum distance as $\mathcal{C}$, in particular it is an $[n',k,d]_{q^m/q}$ code. Moreover, we can also consider the associated $[n',k,d]_{q^m/q}$ system.

\begin{remark}
    From now on, we will consider every $[n,k]_{q^m/q}$ code with effective length $n'<n$ to have generator matrix of the form \eqref{eq:genmatrix}. In particular, we will call the associated $[n',k]_{q^m/q}$ system $U=\langle G^1,\ldots,G^{n'} \rangle_{\Fq}$. 
    \medskip
    
    Moreover, let us introduce the following notation. For every positive integer $s$ we denote by
$$\mathcal{H}_s(q^m)=\{H\subseteq\Fm^s:\dim_{\Fm}(H)=s-1\}$$
    the set of all the hyperplanes of $\Fm^s$. In particular, fix $P=\langle(0,\ldots,0,1)\rangle_{\Fm}$, we need to introduce two subsets of $\mathcal{H}_{k+1}(q^m)$, that are
\begin{align*}
    \mathcal{H}_P^+=\{H\in\mathcal{H}_{k+1}(q^m):P\in H\} \qquad \text{and} \qquad \mathcal{H}_P^-=\{H\in\mathcal{H}_{k+1}(q^m):P\notin H\}.
\end{align*}
Observe that we can identify the elements of $\mathcal{H}_P^+$ and $\mathcal{H}_k(q^m)$ in the following way. There exists a natural isomorphism between $S=\Fm^k\times\{0\}\subseteq\Fm^{k+1}$ and $\Fm^k$, that is the restriction to $S$ of the natural projection $\pi:\Fm^{k+1}\rightarrow\Fm^k$. Using this isomorphism we can construct a bijection between $\mathcal{H}_P^+$ and $\mathcal{H}_k(q^m)$. Define the map
$$\Phi:H\in\mathcal{H}_P^+\longmapsto \pi_{|S}(H\cap S)\in\mathcal{H}_k(q^m),$$
by Grassmann $\dim(H\cap S)=k-1$, hence $\pi_{|S}(H\cap S)\in\mathcal{H}_k(q^m)$. Conversely 
$$\Phi^{-1}:H\in \mathcal{H}_k(q^m)\longmapsto \pi_{|S}^{-1}(H)+P\in\mathcal{H}_P^+.$$

\end{remark}

\begin{definition}\label{def:covering_lift}
    Let $U$ be an $[n',k]_{q^m/q}$ system and let $0\le\delta\le m$
    be an integer. An $[n'+\delta,k+1]_{q^m/q}$ system
    $\widetilde U$ is a
    \textbf{$\delta$-covering lift} of $U$ if
    $\pi(\widetilde U)=U,$ where $\pi:\Fm^{k+1}\rightarrow\Fm^k$ denotes the natural projection on the first $k$ coordinates. We define  $\mathcal L_\delta(U)$ to be the family of all $\delta$-covering lifts of $U$. In particular, for every $\delta$-covering lift $\widetilde U$, we define
    $$\mu_P(\widetilde U)=
    \max_{H\in\mathcal{H}_P^-}\mathrm{wt}_{\widetilde U}(H).$$
\end{definition}

The weights of the hyperplanes in $\mathcal{H}_P^+$ are determined by $U$
and $\delta$, as the following lemma shows.

\begin{lemma}\label{lem:cl_weights}
    Let $\widetilde U$ be a $\delta$-covering lift of an
    $[n',k]_{q^m/q}$ system $U$. Then
    $\mathrm{wt}_{\widetilde U}(P)=\delta$
    and, for every $H\in\mathcal{H}_P^+$,
    $$\mathrm{wt}_{\widetilde U}(H)
    =\mathrm{wt}_U(\Phi(H))+\delta.$$
\end{lemma}

\begin{proof}
    First observe that the restriction of $\pi$ to $\widetilde U$ is an $\Fq$-linear surjection onto $U$ with kernel $\widetilde U\cap P$. Hence,
    \begin{equation}\label{eq:lemlift}
        \wt_{\widetilde U}(P)=\dim_{\Fq}(\widetilde U\cap P)=\dim_{\Fq}(\ker(\pi_{|\widetilde U}))=\dim_{\Fq}(\widetilde U)-\dim_{\Fq}(U)=\delta.
    \end{equation}
    Let $H\in\mathcal{H}_P^+$, since $H=\pi^{-1}(\Phi(H))$,
    the restriction
    $\pi:\widetilde U\cap H\longrightarrow U\cap\Phi(H)$
    is surjective and has kernel $\widetilde U\cap P$.
    Thus, by the rank-nullity theorem and \eqref{eq:lemlift}, we have
    $$
    \mathrm{wt}_{\widetilde U}(H)=\dim_{\Fq}(\widetilde U\cap H)
    =\dim_{\Fq}(U\cap\Phi(H))+\dim_{\Fq}(\widetilde U\cap P)=\mathrm{wt}_U(\Phi(H))+\delta.
    $$
    
\end{proof}

\begin{theorem} \label{thm:hyp}
   Let $m,n$ and $n'$ be positive integers, let $\delta=n-n'$ and assume that $\delta\le m$. Let $\mathcal C$ be an $[n,k]_{q^m/q}$ code with effective length $n'$ and let $U$ be the $[n',k]_{q^m/q}$ system associated with $\mathcal{C}$. Then
   $$\rho(\mathcal{C})=n-\min_{\widetilde U\in\mathcal{L}_{\delta}(U)}\mu_P(\widetilde U).$$
\end{theorem}
\begin{proof}
First, let us prove that for every $1\le t\le\min\{m,n\}$, 
\begin{equation} \label{cond:lift1}
    \rho(\mathcal{C})\ge t\quad\text{if and only if}\quad \exists\;\widetilde U\in\mathcal{L}_{\delta}(U)\quad\text{with}\quad\mu_P(\widetilde U)\le n-t.
\end{equation}
 By Lemma \ref{Lemmacr}, the code $\mathcal C$ has covering radius at least $t$ if and only if there exists $\alpha\in\Fm^n$ such that $\drk(\alpha,\mathcal C)\geq t$. Since $t\geq 1$, necessarily $\alpha\notin\mathcal C$ and $G_\alpha$ has $\Fm$-row rank $k+1$. Thus, $G_\alpha$ is a generator matrix of an $[n,k+1]_{q^m/q}$ code $\mathcal C_\alpha$ satisfying
\begin{equation}\label{eqLemma}
     \min\{\rk(u'G_{\alpha}):u'=(u_1,\ldots,u_k,u_{k+1})\in\Fm^{k+1}, u_{k+1}\neq0\}\geq t.
\end{equation}
Since $u_{k+1}\neq0$, we can normalize the vector $u'$ and consider $u'=(u_1,\ldots,u_k,1)$ by redefining $u_i$ as the ratio ${u_i}/{u_{k+1}}$, for $i\in[k]$. From \eqref{eqLemma} we get 

$$\rk\left((u_1,\ldots,u_k,1)G_{\alpha}\right)\geq t, \quad\text{for every} \;(u_1,\ldots,u_k)\in\Fm^k.$$ 

In order to construct the $[n,k+1]_{q^m/q}$ system associated with $\mathcal{C}_{\alpha}$, we need to prove that we can always take an $\alpha\in\Fm^n$ such that $\mathcal{C}_{\alpha}$ is nondegenerate. By definition,  $\mathcal{C}_{\alpha}$ is nondegenerate if and only if the columns of $G_{\alpha}$ are $\Fq$-independent. The first $n'$ columns of $G_{\alpha}$ are $\Fq$-independent since we have taken a generator matrix for $\mathcal{C}$ of the form \eqref{eq:genmatrix}. Hence, $\mathcal{C}_{\alpha}$ is degenerate if and only if $\alpha_{n'+1},\ldots,\alpha_n$ are $\Fq$-dependent, i.e. there exists a nonnegative integer $s$ such that 
$$s=\dim_{\Fq}(\langle \alpha_{n'+1},\ldots,\alpha_n\rangle_{\Fq})<\delta.$$
Up to permuting the last $\delta$ coordinates of $\alpha$,
we can assume that
$\alpha_{n'+1},\ldots,\alpha_{n'+s}$ form an $\Fq$-basis
of $\langle\alpha_{n'+1},\ldots,\alpha_n\rangle_{\Fq}$. Since $m\ge \delta$, there exist $z_{n'+s+1},\ldots,z_n\in\Fm$ such that $\alpha_{n'+1},\ldots,\alpha_{n'+s},z_{n'+s+1},\ldots,z_n$ are $\Fq$-independent. Define 
$\alpha'=(\alpha_1,\ldots,\alpha_{n'+s},z_{n'+s+1},\ldots,z_n)\in\Fm^n.$
Observe that $$\langle\alpha_{n'+1},\ldots,\alpha_n\rangle_{\Fq}\subseteq \langle \alpha_{n'+1},\ldots,\alpha_{n'+s},z_{n'+s+1},\ldots,z_n\rangle_{\Fq},$$
 and every $c\in\mathcal{C}$ has the form $c=(c_1,\ldots,c_{n'},0,\ldots,0)$, hence $\drk(\alpha',c)\ge \drk(\alpha,c)$ for every $c\in\mathcal{C}$. Thus, it holds also for the minimum over all the codewords, i.e.
$$\drk(\alpha',\mathcal{C})\geq\drk(\alpha,\mathcal{C})\geq t. $$
When $m\ge \delta$, replacing $\alpha$ with $\alpha'$ if necessary, we can assume that $\mathcal{C}_{\alpha}$ is nondegenerate and $\drk(\alpha,\mathcal C)\ge t$.

Since $\mathcal{C}_{\alpha}$ is an $[n,k+1]_{q^m/q}$ nondegenerate code, we can consider the associated $[n,k+1]_{q^m/q}$ system $U_{\alpha}$. First observe that $\pi(U_\alpha)=\langle G^1,\ldots,G^{n'}\rangle_{\Fq}=U$, hence $U_\alpha\in\mathcal{L}_{\delta}(U)$ and by Lemma \ref{lem:cl_weights} we have $$\wt_{U_{\alpha}}(H)=\wt_{U}(\Phi(H))+\delta,$$ for every hyperplane $H\in\mathcal{H}_P^+$.
Moreover, for every $u'\in\Fm^{k+1}$ we have that $\rk(u'G_{\alpha})=n-\wt_{U_{\alpha}}( \langle u'\rangle_{\Fm}^{\perp})$ and
$$\mathcal{H}_P^-=\{H\in\mathcal{H}_{k+1}(q^m):P\notin H\}=\{\langle(u_1,\ldots,u_k,1)\rangle_{\Fm}^{\perp} : (u_1,\ldots,u_k)\in\Fm^k\}.$$
Hence, by \eqref{eqLemma} we get
\begin{equation*}
    n-\wt_{U_{\alpha}}( \langle(u_1,\ldots,u_k,1)\rangle_{\Fm}^{\perp})=\rk((u_1,\ldots,u_k,1)G_{\alpha})\geq t,
\end{equation*}
and this means that $\wt_{U_{\alpha}}(H)\leq n-t$ for every $H\in\mathcal{H}_P^-$, hence $\mu_P(U_\alpha)\le n-t$. 

Conversely, assume there exists $\widetilde U\in\mathcal{L}_{\delta}(U)$ such that $\mu_P(\widetilde U)\le n-t$. Since $\pi(\widetilde U)=U$, for every $i\in [n']$ there exists $\alpha_i\in\Fm$ such that $(G^i,\alpha_i)\in\widetilde U$. By Lemma \ref{lem:cl_weights} $\dim_{\Fq}(\widetilde U\cap P)=\delta$, hence there exist $\alpha_{n'+1},\ldots,\alpha_n\in\Fm$ linearly independent over $\Fq$ such that $\{(0,\ldots,0,\alpha_{n'+1}),\ldots,(0,\ldots,0,\alpha_n)\}$ form a basis for $\widetilde U\cap P$. Since the first $n'$ columns of $G$ are $\Fq$-independent and also $\alpha_{n'+1},\ldots,\alpha_n$ are $\Fq$-independent, we have found a basis of $\widetilde U$ and 
$$\widetilde U=\langle(G^1,\alpha_1),\ldots,(G^{n'},\alpha_{n'}),(0,\alpha_{n'+1}),\ldots,(0,\alpha_n) \rangle_{\Fq},$$
where $0$ denotes the zero-vector in $\Fm^k$. For every $H=\langle(u,1)\rangle_{\Fm}^\perp\in\mathcal{H}_P^-$, with $u\in\Fm^k$, we have
$$\rk(uG+\alpha)=n-\wt_{\widetilde U}(H)\ge n-\mu_P(\widetilde U)\ge t,$$ hence $\rho(\mathcal{C})\ge t$. Let us prove that \eqref{cond:lift1} implies the statement. Applying \eqref{cond:lift1} with $t=\rho(\mathcal C)$ we get a lift $\widetilde U_0\in\mathcal L_\delta(U)$ such that $$\min_{\widetilde U\in\mathcal L_\delta(U)}\mu_P(\widetilde U)\le\mu_P(\widetilde U_0)\le n-\rho(\mathcal C).$$
On the other hand, the preceding construction
associates with every lift $\widetilde U\in\mathcal L_\delta(U)$ a vector $\alpha$ such that $\drk(\alpha,\mathcal C)=n-\mu_P(\widetilde U)$. Hence,
$$\mu_P(\widetilde U)=n-\drk(\alpha,\mathcal C)\ge n-\rho(\mathcal C),$$
for every $\widetilde U\in\mathcal L_\delta(U)$, in particular it holds for the minimum over all the possible $\delta$-covering lifts and the statement follows.

\end{proof}

\begin{lemma} \label{lemma:uppb}
 Let $\mathcal{C}$ be an $[n,k]_{q^m/q}$ code with effective length $n'$, let $1\le t\le\min\{m,n\}$ be a positive integer and let $\delta=n-n'$. If $\rho(\mathcal{C})\ge t$, then 
 \begin{equation} \label{eq:lemmauppb}
     q^n-q^{\delta}\le q^m(q^{n-t}-1).
 \end{equation}
\end{lemma}

\begin{proof} If $\delta\ge m$, then immediately
$$q^m(q^{n-t}-1)=q^{m+n-t}-q^m\ge q^n-q^m\ge q^n-q^{n-n'}=q^n-q^\delta.$$
Assume now that $\delta<m$. Since $\rho(\mathcal{C})\ge t$ we can apply \eqref{cond:lift1} of Theorem \ref{thm:hyp} and there exists $\widetilde U\in\mathcal{L}_\delta(U)$ such that $\mu_P(\widetilde U)\le n-t$. Thus, $\wt_{\widetilde U}(H)\leq n-t$ for every hyperplane $H\in\mathcal{H}_P^-$, while by Lemma \ref{lem:cl_weights} $\wt_{\widetilde U}(H)=\wt_{U}(\Phi(H))+\delta$ for every $H\in\mathcal{H}_P^+$. Hence, we can compute the cardinality of the intersection of $\widetilde U\setminus\{0\}$ with every hyperplane of $\Fm^{k+1}$. Since every $H\in\mathcal{H}_P^-$ has the form $\langle(u_1,\ldots,u_k,1)\rangle_{\Fm}^\perp$, with $u_1,\ldots,u_k\in\Fm$, $|\mathcal{H}_P^-|=q^{mk}$, then
$$\sum_{H\in\mathcal{H}_P^-}\left|(\widetilde U\setminus\{0\})\cap H\right|\le q^{mk}(q^{n-t}-1).$$
For $H\in\mathcal{H}_P^+$ we have that
$$| \widetilde U\cap H|=q^{\delta}| U\cap \Phi(H)|,$$
hence
$$| (\widetilde U\setminus\{0\})\cap H|=q^{\delta}(| (U\setminus\{0\})\cap \Phi(H)|+1)-1.$$
If we sum on all the hyperplanes $H\in\mathcal{H}_P^+$, since $\Phi$ is a bijection, we get
\begin{align*}
    \sum_{H\in\mathcal{H}_P^+}\left|(\widetilde U\setminus\{0\})\cap H\right|&=\sum_{K\in\mathcal{H}_k(q^m)}\left(q^{\delta}(\left|(U\setminus\{0\})\cap K\right|+1)-1\right)\\
    &=q^{\delta}\sum_{K\in\mathcal{H}_k(q^m)}\left|(U\setminus\{0\})\cap K\right| + (q^{\delta}-1)\binom{k}{1}_{q^m} \\
    &=q^{\delta}(q^{n'}-1)\binom{k-1}{1}_{q^m} + (q^{\delta}-1)\binom{k}{1}_{q^m},
\end{align*}
where if $k=1$ we assume $\binom{0}{1}_{q^m}=0$. By \cite[Lemma 3.6]{ABNR2022} the sum of the intersection of $\widetilde U\setminus\{0\}$ with all the hyperplanes of $\Fm^{k+1}$ is
$$\sum_{H\in\mathcal{H}_{k+1}(q^m)}|(\widetilde U\setminus\{0\})\cap H|=(q^n-1)\binom{k}{1}_{q^m}.$$
Since 
$$\sum_{H\in\mathcal{H}_{k+1}(q^m)}|(\widetilde U\setminus\{0\})\cap H|=\sum_{H\in\mathcal{H}_P^+}|(\widetilde U\setminus\{0\})\cap H|+\sum_{H\in\mathcal{H}_P^-}|(\widetilde U\setminus\{0\})\cap H|,$$
we get
\begin{equation*} 
    (q^n-1)\binom{k}{1}_{q^m}\leq (q^{mk+n-t}-q^{mk})+(q^n-q^\delta)\binom{k-1}{1}_{q^m}+(q^\delta-1)\binom{k}{1}_{q^m}.
\end{equation*}
Rearranging the terms we get 
\begin{equation*} 
    (q^n-q^\delta)\left(\binom{k}{1}_{q^m}-\binom{k-1}{1}_{q^m}\right)\leq q^{mk}(q^{n-t}-1).
\end{equation*}
Since $\left(\binom{k}{1}_{q^m}-\binom{k-1}{1}_{q^m}\right)=q^{m(k-1)}$, we get \eqref{eq:lemmauppb}.

\end{proof}

\begin{proposition} \label{prop:uppb}
    Let $\mathcal{C}$ be an $[n,k]_{q^m/q}$ code with effective length $n'$ and let $\delta=n-n'$. Then $\rho(\mathcal{C})=m$ if and only if $\delta\geq m$. In particular, if $\delta<m$, then $\rho(\mathcal{C})\leq\min\{m,n\}-1.$
\end{proposition}
\begin{proof}

Assume that $\delta\ge m$. Recall that every $c\in\mathcal{C}$ has the form $c=(c_1,\ldots,c_{n'},0,\ldots,0)$. Take $x=(0,\ldots,0,x_{n'+1},\ldots,x_{n'+m},0,\ldots,0)$ where $x_{n'+1},\ldots,x_{n'+m}\in\Fm$ are $\Fq$-independent. Since for every $c\in\mathcal{C}$ the vector $x-c$ has at least $m$ coordinates $\Fq$-independent, $\drk(x,c)\ge m$ for every $c\in\mathcal{C}$. Hence, $\drk(x,\mathcal{C})\ge m$ and  $\rho(\mathcal{C})\ge m$, but the covering radius cannot exceed $m$, so $\rho(\mathcal{C})=m$.

Let $m\ge n$ and assume by contradiction that $\rho(\mathcal{C})\ge n$, then by Lemma \ref{lemma:uppb}
\begin{equation*}
    q^n-q^\delta\le q^m(q^{n-n}-1)=0.
\end{equation*}
Since the code is nonzero, $n'>0$, hence $q^n-q^\delta>0$ and we get a contradiction. We can conclude that $\rho(\mathcal{C})\le n-1$.

Let $m-n'<\delta<m$ or, equivalently, $m<n< m+n'$ and assume by contradiction that $\rho(\mathcal{C})\ge m$. By Lemma \ref{lemma:uppb}
\begin{equation*}
    q^n-q^\delta\le q^m(q^{n-m}-1)=q^n-q^m,
\end{equation*}
hence $q^\delta\ge q^m$ but this is a contradiction since by hypothesis $\delta<m$. Thus, in this case $\rho(\mathcal{C})\le m-1$.

\end{proof}

In order to apply the characterization obtained in Theorem \ref{thm:hyp}, it is useful to compare codes of length $n\ge m$ with suitable puncturings of length $m$, as shown in the following proposition.

\begin{proposition} \label{prop:tagli}
    Let $n\ge m$ and let $r$ be a positive integer such that $k\le r\le m$. Assume that every $[m,k]_{q^m/q}$ code with effective length $r$ has covering radius $m-1$. Then, for $s\in\{0,\ldots,n-m\}$ every $[n,k]_{q^m/q}$ code with effective length $s+r$ has covering radius at least $m-1$.
\end{proposition}
\begin{proof}
    Let $\mathcal{C}$ be an $[n,k]_{q^m/q}$ code with effective length $r+s$ and assume it has generator matrix $G$ of the form \eqref{eq:genmatrix}. Since $G$ has $\Fm$-row rank $k$, its columns span $\Fm^k$ and among $G^1,\ldots,G^{r+s}$, we can choose $k$
    columns which are $\Fm$-independent. Since $r\geq k$, we can
    extend this set to a set $\mathcal J=\{j_1,\ldots,j_r\}\subseteq\{1,\ldots,r+s\}$ of cardinality $r$. Moreover, since $s\leq n-m$, we have $n-(r+s)\geq m-r$. Hence, among the zero columns of $G$, we can choose a set $\mathcal Z\subseteq\{r+s+1,\ldots,n\}$ of cardinality $m-r$. Let $\mathcal{I}=[n]\setminus (\mathcal{J}\cup \mathcal Z)$, then $|\mathcal{I}|=n-m$. Consider the punctured code
    \begin{equation*}
        \Pi(\mathcal{C,\mathcal{I}})=\mathrm{rowsp}\begin{pmatrix}
            
G^{j_1} \;\, \dots \;\; G^{j_r} \;\; 0 \;\;\dots \;\; 0
            
        \end{pmatrix}.
    \end{equation*}
     The columns $G^{j_1},\ldots,G^{j_r}$ are $\Fq$-independent,
    since they are a subset of $G^1,\ldots,G^{r+s}$. Thus,
    $\Pi(\mathcal C,\mathcal I)$ has effective length $r$.
    Moreover, these columns contain $k$ columns which are $\Fm$-independent. Hence, this matrix has $\Fm$-row rank $k$ and thus $\Pi(\mathcal{C,\mathcal{I}})$ is an $[m,k]_{q^m/q}$ code with effective length $r$. By hypothesis $\rho(\Pi(\mathcal{C,\mathcal{I}}))= m-1$. Hence, by Lemma \ref{Lemmacr}, there exists $\alpha=(\alpha_1,\ldots,\alpha_m)\in\Fm^m$ such that 
    $$\drk((u G)^{[n]\setminus\mathcal{I}},\alpha)=\rk\left((u G)^{[n]\setminus\mathcal{I}}-\alpha\right)\ge m-1 \quad\text{for every } u\in\Fm^k.$$
Let $\alpha'\in\Fm^n$ such that $\pi_{\mathcal{I}}(\alpha')=\alpha$, then
$$\drk(uG,\alpha')\ge\drk((uG)^{[n]\setminus\mathcal{I}},\alpha)\ge m-1\quad\text{ for every } u\in\Fm^k.$$
Thus, $\rho(\mathcal{C})\ge\rho(\Pi(\mathcal{C},\mathcal{I}))\ge m-1$.

\end{proof}

\section{Covering radius of $[n,1]_{q^m/q}$ codes} \label{sec:k1}
In this section, we focus on $[n,1]_{q^m/q}$ codes. We specialize the results obtained in the previous section, and in particular Theorem \ref{thm:hyp}, to the case of $1$-dimensional rank-metric codes. Since in this setting the covering lifts lie in $\Fm^2$, weight conditions on hyperplanes translate into weight conditions on points. This leads to a connection with a well-studied class of $\Fq$-subspaces known as $t$-clubs. Unlike the Hamming-metric case, $1$-dimensional rank-metric codes with the same length and minimum distance are not necessarily equivalent. Indeed, in the Hamming metric, any two $[n,1,d]_q$ codes are equivalent once $n$ and $d$ are fixed, thus their covering radius can be determined by studying a single representative. In the rank-metric setting, however, there exist $[n,1,d]_{q^m/q}$ codes with the same length and minimum distance but with different covering radii.

\begin{remark} \label{rem:n'd}
    As in the general case, for $k=1$ we can consider the notion of effective length. However, since a generator matrix for an $[n,1]_{q^m/q}$ code is an element of $\Fm^{1\times n}$, the columns are elements of $\Fm$ and the effective length coincides with the minimum distance of the code. Indeed, let $\mathcal{C}$ be an $[n,1,d]_{q^m/q}$ code and let $(v_1,\ldots,v_n)\in\Fm^{1\times n}$ be a generator matrix for $\mathcal{C}$. The effective length is
    $$n'=\dim_{\Fq}(\langle v_1,\ldots,v_n\rangle_{\Fq})=\rk(v_1,\ldots,v_n)=d.$$
    As in the previous section, there exists $A\in\mathrm{GL}(n,q)$ such that $(v_1,\ldots,v_n)A=(\tilde{v}_1,\ldots,\tilde{v}_d,0,\ldots,0)$. Since this operation preserves the covering radius, from now on we can assume that every $[n,1,d]_{q^m/q}$ code, up to linear equivalence, is generated by a vector of the form $(v_1,\ldots,v_d,0,\ldots,0)$, where $v_1,\ldots,v_d\in\Fm$ are $\Fq$-independent.
    
\end{remark}

\begin{lemma}\label{lem:equivalence1d}
Let $ \mathcal C=\left\langle(v_1,\ldots,v_d,0,\ldots,0)\right\rangle_{\Fm}
$ and
$\mathcal C'=\left\langle(w_1,\ldots,w_d,0,\ldots,0)\right\rangle_{\Fm}$
be two $[n,1,d]_{q^m/q}$ codes. Let
$V=\langle v_1,\ldots,v_d\rangle_{\Fq}$ and $V'=\langle w_1,\ldots,w_d\rangle_{\Fq}.$
Then
$\mathcal C\sim\mathcal C'$ if and only if  $V'=\lambda V$ for some $\lambda\in\Fm^*$.

\end{lemma}

\begin{proof}
Let $v=(v_1,\ldots,v_d,0,\ldots,0)$ and $w=(w_1,\ldots,w_d,0,\ldots,0).$ Suppose first that $\mathcal C\sim\mathcal C'$. Then there exists
$A\in\mathrm{GL}(n,q)$ such that
$\mathcal C'=\mathcal CA.$
Since $ \langle w\rangle_{\Fm}=\langle vA\rangle_{\Fm}$, there exists $\lambda\in\Fm^*$ such that
$ w=\lambda vA.$ Since $A\in\mathrm{GL}(n,q)$, the coordinates of $vA$ are $\Fq$-linear combinations of the coordinates of $v$, hence
$\left\langle(vA)_1,\ldots,(vA)_n\right\rangle_{\Fq}\subseteq V,$
but also $ v=(vA)A^{-1}$ and the
reverse inclusion holds. Hence,
$$V'=\langle w_1,\ldots,w_d\rangle_{\Fq}=\lambda\left\langle( vA)_1,\ldots,( vA)_n\right\rangle_{\Fq}=\lambda V.$$
Conversely, suppose that
$V'=\lambda V$ for some $\lambda\in\Fm^*$. Therefore, $w_1,\ldots,w_d$ and
$\lambda v_1,\ldots,\lambda v_d$ are two $\Fq$-bases of
$V'$ and there exists
$B\in\mathrm{GL}(d,q)$ such that
$(w_1,\ldots,w_d)=\lambda(v_1,\ldots,v_d)B.$
If we consider the matrix
$$
A=
\begin{pmatrix}
B & 0\\
0 & I_{n-d}
\end{pmatrix}
\in\mathrm{GL}(n,q),
$$
then $w=\lambda vA.$ Since multiplication by the nonzero scalar $\lambda$ does not change the $\Fm$-span of a vector, we obtain
$$\mathcal C'=\langle w\rangle_{\Fm}=\langle\lambda vA\rangle_{\Fm}=\langle vA\rangle_{\Fm}=\mathcal CA,$$
and the statement holds.

\end{proof}

Thus, the equivalence class of a $1$-dimensional code $\mathcal C$ is completely determined by the orbit of $V$ under scalar multiplication by $\Fm^*$.

\medskip

Throughout this section, the point $P$ introduced above specializes to $P=\langle(0,1)\rangle_{\Fm}$.
We specialize Theorem \ref{thm:hyp} and Proposition \ref{prop:uppb} to $k=1$, recalling that in this case, as just observed in Remark \ref{rem:n'd} the effective length coincides with the minimum distance.

\begin{corollary} \label{cor:k1}
   Let $m,n$ and I $d$ be positive integers and let $\delta=n-d$ such that $\delta\le m$. Let $\mathcal C=\langle (v_1,\ldots,v_d,0,\ldots,0)\rangle_{\Fm}$ be an $[n,1,d]_{q^m/q}$ code and let $V=\langle v_1,\ldots,v_d\rangle_{\Fq}$. Then 
   $$\rho(\mathcal{C})=n-\min_{\widetilde U\in\mathcal{L}_\delta(V)}\mu_P(\widetilde U).$$
   
\end{corollary}

\begin{corollary} \label{cor:uppb}
    Let $\mathcal{C}$ be an $[n,1,d]_{q^m/q}$ code and let $\delta=n-d$. Then $\rho(\mathcal{C})=m$ if and only if $\delta\ge m$. In particular, if $\delta<m$, then $\rho(\mathcal{C})\le\min\{m,n\}-1$.
\end{corollary}

\subsection{Codes attaining the upper bound}
When $m\ge n$ Corollary \ref{cor:k1} provides a correspondence between codes with  covering radius $n-1$ and clubs. When $m<n$ clubs give sufficient conditions for the covering radius to be at least $m-1$. This new geometric perspective can help in determining the covering radius of $1$-dimensional codes.

\begin{corollary} \label{cor:corrclub}
    Let $\mathcal{C}=\langle (v_1,\ldots,v_d,0,\ldots,0)\rangle_{\Fm}$ be an $[n,1,d]_{q^m/q}$ code with $m\geq n$, let $V=\langle v_1,\ldots,v_d\rangle_{\Fq}$ and let $\delta=n-d$. Then $\rho(\mathcal{C})=n-1$  if and only if there exists $\widetilde U\in\mathcal{L}_\delta(V)$ which is a $\delta$-club of rank $n$ with $\langle(0,1)\rangle_{\Fm}$ of weight $\delta$.
\end{corollary}
\begin{proof}
Since $m\ge n$, by Corollary \ref{cor:uppb} we have $\rho(\mathcal{C})\le \min\{m,n\}-1=n-1$.
    Moreover, by Corollary \ref{cor:k1} we have that $\rho(\mathcal{C})\geq n-1$ and hence the equality, if and only if there exists $\widetilde U\in\mathcal{L}_{\delta}(V)$ such that $\mu_P(\widetilde U)\le 1$.
    By Lemma \ref{lem:cl_weights}, $\wt_{\widetilde U}(\langle(0,1)\rangle_{\Fm})=\delta$, hence $\widetilde U$ is a $\delta$-covering lift which is a $\delta$-club of rank $n$.

\end{proof}

\begin{corollary} \label{cor:n>m}
    Let $m,n$ and $d$ be positive integers, let $\delta=n-d$ such that $m-d<\delta<m$, let $\mathcal{C}=\langle (v_1,\ldots,v_d,0,\ldots,0)\rangle_{\Fm}$ be an $[n,1,d]_{q^m/q}$ code and let $V=\langle v_1,\ldots,v_d\rangle_{\Fq}$. If there exists $\widetilde U\in\mathcal
    L_{m-d}(V)$ which is an $(m-d)$-club of rank $m$ with $\wt_{\widetilde U}(\langle(0,1)\rangle_{\Fm})=m-d$, then $\rho(\mathcal C)=m-1$.  
\end{corollary}
\begin{proof}
    Let $\widetilde U$ be an $(m-d)$-covering lift that is an $(m-d)$-club of rank $m$ and let us write it in coordinates. Since it is a covering lift, there exists $x=(\alpha_1,\ldots,\alpha_d,\beta_1,\ldots,\beta_{m-d})\in\Fm^m$ such that $$\widetilde U=\langle(v_1,\alpha_1),\ldots,(v_d,\alpha_d),(0,\beta_1),\ldots,(0,\beta_{m-d})\rangle_{\Fq}.$$ Since it is an $(m-d)$-covering lift which is an $(m-d)$-club of rank $m$, we have
    \begin{align*}
        &\wt_{\widetilde U}(\langle(0,1)\rangle_{\Fm})=m-d,\\
        &\wt_{\widetilde U}(Q)\leq1 \quad \text{for every } Q\neq\langle(0,1)\rangle_{\Fm},
    \end{align*}
    and $\beta_1,\ldots,\beta_{m-d}$ are linearly independent over $\Fq$.
    Since $\delta<m$ there exist $\beta_{m-d+1},\ldots,\beta_{\delta}\in\Fm$ linearly independent over $\Fq$ such that $\beta_1,\ldots,\beta_{\delta}$ are $\Fq$-independent, define $y=(\alpha_1,\ldots,\alpha_d,\beta_1,\ldots,\beta_{\delta})$ and the $[n,2]_{q^m/q}$ system 
    \begin{equation*}
        U_y=\langle(v_1,\alpha_1),\ldots,(v_d,\alpha_d),(0,\beta_{1}),\ldots,(0,\beta_{\delta})\rangle_{\Fq}\supseteq \widetilde U.
    \end{equation*}
 Clearly $ \widetilde U\subseteq U_y$ and $\pi(U_y)=V$, hence $U_y\in\mathcal L_\delta(V)$. If we compute the weight of the points we get
\begin{align*}
        &\wt_{U_y}(\langle(0,1)\rangle_{\Fm})=\dim_{\Fq}(\langle\beta_1,\ldots,\beta_\delta\rangle_{\Fq})=\delta,\\
        &\wt_{U_y}(Q)\leq n-m+1 \quad \text{for every } Q\neq\langle(0,1)\rangle_{\Fm}.
    \end{align*}
By Corollary \ref{cor:k1}
$$\rho(\mathcal C)=n-\min_{\widetilde U'\in\mathcal L_\delta(V)}
       \mu_{P}(\widetilde U')\geq n-\mu_{P}(U_y)\geq m-1.$$
On the other hand, since $\delta<m$, Corollary \ref{cor:uppb} gives
$\rho(\mathcal C)\leq m-1$, thus the equality holds.
    
\end{proof}

Let us write Proposition \ref{prop:tagli} for $k=1$.

\begin{corollary} \label{cor:taglik1}
    Let $n\ge m$, let $r$ be a positive integer such that $1\le r\le m$ and assume that every $[m,1,r]_{q^m/q}$ code has covering radius $m-1$. Then, for $s\in\{0,\ldots,n-m\}$ every $[n,1,s+r]_{q^m/q}$ code has covering radius at least $m-1$.
\end{corollary}

\begin{proposition}\label{prop:caractd}
    Let
    $\mathcal C=\langle(v_1,\ldots,v_d,0,\ldots,0)\rangle_{\Fm}$
    be an $[n,1,d]_{q^m/q}$ code, let
    $V=\langle v_1,\ldots,v_d\rangle_{\Fq}$ and let $\delta=n-d$.
    Assume that $\delta<m$.
    Then the following statements are equivalent:
    \begin{enumerate}
        \item[(i)] $\rho(\mathcal C)=\min\{m,n\}-1$.

        \item[(ii)] There exist an $\Fq$-subspace $W$ of $\Fm$ of
        dimension $\delta$ and an $\Fq$-linear map
        $f:V\longrightarrow\Fm$ such that
        \begin{equation}\label{cond:d:dim}
            \dim_{\Fq}
            (\{x\in V:f(x)-\lambda x\in W\})
            \leq n-\min\{m,n\}+1,
            \qquad\forall\,\lambda\in\Fm.
        \end{equation}

        \item[(iii)] There exist an $\Fq$-subspace $W$ of $\Fm$ of
        dimension $\delta$ and an $\Fq$-linear map
        $f:V\longrightarrow\Fm$ such that, denoting by
        $\pi_W:\Fm\longrightarrow\Fm/W$ the canonical quotient map,
        \begin{equation}\label{cond:d}
            \dim_{\Fq}(\mathrm{Im}
            \left(
                \pi_W\circ
                (f-\lambda\,\mathrm{id}_V)
            \right))
            \geq \min\{m,n\}-\delta-1,
            \qquad\forall\,\lambda\in\Fm.
        \end{equation}
    \end{enumerate}
    
\end{proposition}

\begin{proof}
    Let
    $v=(v_1,\ldots,v_d,0,\ldots,0).$ Since $\mathcal C$ has minimum rank distance $d$, the elements
    $v_1,\ldots,v_d$ are linearly independent over $\Fq$ and hence
    form an $\Fq$-basis of $V$. By Corollary \ref{cor:uppb}, the assumption $\delta<m$ gives $\rho(\mathcal C)
        \leq\min\{m,n\}-1.$ Moreover, let $P_\lambda=\langle(1,\lambda)\rangle_{\Fm}$,
        for every $\lambda\in\Fm$, which are exactly all the points different from $P$.

    \medskip

    \underline{$(i)\Rightarrow(ii).$}
    Suppose that $\rho(\mathcal C)=\min\{m,n\}-1$. By Corollary \ref{cor:k1}, there exists $\widetilde U\in\mathcal
    L_{\delta}(V)$ with $\mu_P(\widetilde U)\le n-\min\{m,n\}+1$. Hence,
    $\mathrm{wt}_{\widetilde U}(P)=\delta$
    and $\mathrm{wt}_{\widetilde U}(P_\lambda)
        \leq n-\min\{m,n\}+1$, for every
        $\lambda\in\Fm.$ Since $\widetilde U$ is a $\delta$-covering lift, there exists $\alpha=(\alpha_1,\ldots,\alpha_n)\in\Fm^n$ such that $\widetilde U=\langle(v_1,\alpha_1),\ldots,(v_n,\alpha_n)\rangle_{\Fq}$. Define
    $W=\langle \alpha_{d+1},\ldots,\alpha_n\rangle_{\Fq}$, so that $\widetilde U\cap P=\{0\}\times W$ and let $f:V\longrightarrow\Fm$ be the unique $\Fq$-linear map satisfying
    $f(v_i)=\alpha_i,$ for every $i\in[d].$
    Then
    $$\widetilde U= \{(x,f(x)+w):x\in V,\ w\in W\}.$$
    In particular, since $\mathrm{wt}_{\widetilde U}(P)=\delta$, we have 
    $\dim_{\Fq}(W)=\delta.$ Moreover, for every $\lambda\in\Fm$, observe that $(x,\lambda x)\in \widetilde U$ if and only if $x\in V$ and $\lambda x=f(x)+w$, i.e. $f(x)-\lambda x=-w\in W$, thus
    $$\widetilde U\cap P_\lambda=\left\{(x,\lambda x):x\in V,\ f(x)-\lambda x\in W\right\}.$$
    Thus,
    $$\dim_{\Fq}(\{x\in V:f(x)-\lambda x\in W\})=\mathrm{wt}_{\widetilde U}(P_\lambda)\le n-\min\{m,n\}+1,$$
and \eqref{cond:d:dim} holds.

    \medskip

    \underline{$(ii)\Rightarrow(i).$}
    Let us consider the $\Fq$-subspace 
    $$ \widetilde U=\{(x,f(x)+w):x\in V,\ w\in W\}.$$
    In particular, $\dim_{\Fq}(\widetilde U)=\dim_{\Fq}(V)+\dim_{\Fq}(W) =d+\delta=n.$  Moreover, if we compute the weight of the points we get $\mathrm{wt}_{\widetilde U}(P)=\dim_{\Fq}(W)=\delta,$
    while, for every $\lambda\in\Fm$,
    $$\mathrm{wt}_{\widetilde U}(P_\lambda)=\dim_{\Fq}(\{x\in V:f(x)-\lambda x\in W\})\leq n-\min\{m,n\}+1,$$
    Since $\delta<n$ and $n-\min\{m,n\}+1<n$, $\widetilde U$ is not contained in any point of $\Fm^2$. Hence, $\widetilde U$ is an
$[n,2]_{q^m/q}$ system. Since $\pi(\widetilde U)=V$,
we have that $\widetilde U$ is a $\delta$-covering lift with $\mu_P(\widetilde U)\le n-\min\{m,n\}+1$. Thus, Corollary \ref{cor:k1} yields $\rho(\mathcal C)\geq \min\{m,n\}-1.$ Since we already have that $\rho(\mathcal{C})\le \min\{m,n\}-1$, this gives $\rho(\mathcal C)=\min\{m,n\}-1$.

    \medskip

     \underline{$(ii)\Longleftrightarrow(iii).$}
    For every $\lambda\in\Fm$, we have
    $$\ker\left(\pi_W\circ(f-\lambda\,\mathrm{id}_V)\right)=\{x\in V:f(x)-\lambda x\in W\}.$$
    Since $\dim_{\Fq}(V)=d$, the rank-nullity theorem gives
    $$\dim_{\Fq}(\mathrm{Im}\left( \pi_W\circ (f-\lambda\,\mathrm{id}_V)\right))=d-\dim_{\Fq}(\{x\in V:f(x)-\lambda x\in W\}).$$
    Hence, conditions \eqref{cond:d:dim} and \eqref{cond:d} are equivalent.
    
\end{proof}

\begin{remark}\label{rem:d:evasive}
     Condition \eqref{cond:d:dim} can be read
    in terms of generalised evasive subspaces as in \cite{GRSZ2024}.
    Recall that, given a family $\mathcal{S}$ of $\Fq$-subspaces
    of a vector space $X$, all having the same dimension, an
    $\Fq$-subspace $U$ of $X$ is called $(\mathcal{S},h)$-evasive if
    $$\wt_U(S)\leq h,\quad \forall\;S\in\mathcal{S},$$
    see \cite[Definition 1.2]{GRSZ2024}.
Regard $\Fm^2$ as a $2m$-dimensional vector space over $\Fq$ and let
    $\mathcal{S}=\{\langle(1,\lambda)\rangle_{\Fm}:\lambda\in\Fm\}.$
    Then $\mathcal{S}$ is the partial Desarguesian spread obtained
    by removing the spread element corresponding to the point
    $P=\langle(0,1)\rangle_{\Fm}$ from the Desarguesian spread
    of $\Fm^2$. With the notation of Proposition \ref{prop:caractd}, let $W\subseteq\Fm$ be an $\Fq$-subspace of dimension $\delta=n-d$ and define
    $$U_{f,W}=
    \{(x,f(x)+w):x\in V,\ w\in W\}.$$
Since $\dim_{\Fq}(V)=d$ and $\dim_{\Fq}(W)=\delta$, we have
    $\dim_{\Fq}(U_{f,W})=n$.
Moreover,
    $U_{f,W}\cap P
    =
    \{0\}\times W$ and for every $\lambda\in\Fm$, we  have $$U_{f,W}\cap\langle(1,\lambda)\rangle_{\Fm}=\{(x,\lambda x):x\in V,\ f(x)-\lambda x\in W\}.$$
Hence,
    $$\dim_{\Fq}\left(U_{f,W}\cap\langle(1,\lambda)\rangle_{\Fm}\right)
    =\dim_{\Fq}(\{x\in V:f(x)-\lambda x\in W\}).$$
Consequently, condition \eqref{cond:d:dim} holds if and only if
$U_{f,W}$ is $(\mathcal S,n-\min\{m,n\}+1)$-evasive. Thus, Proposition \ref{prop:caractd} characterizes the equality $\rho(\mathcal{C})=\min\{m,n\}-1$ by the existence of $W$ and $f$ such that the associated subspace $U_{f,W}$ is $(\mathcal{S},n-\min\{m,n\}+1)$-evasive. In particular if $m\ge n$, then $\min\{m,n\}=n$ and $U_{f,W}$ is $(\mathcal S,1)$-evasive.
\end{remark}

Condition \eqref{cond:d} in Proposition \ref{prop:caractd} leads to an exact expression for the covering radius in terms of auxiliary matrix codes. Recall that a (linear) matrix code is an $\Fq$-subspace $\mathcal{C}\subseteq \Fq^{m\times n}$ of the space of $m\times n$ matrices with coefficients in $\Fq$ endowed with the distance $$\drk(A,B)=\rank(A-B),\quad A,B\in\Fq^{m\times n}.$$
    Let $V\subseteq\Fm$ be an $\Fq$-subspace of dimension $d$, let $\delta=n-d$ and assume that $0\leq \delta\leq m$. Fix an $\Fq$-subspace $W$ of $\Fm$
of dimension $\delta$ and, for every $\lambda\in\Fm$, define the $\Fq$-linear map $A_\lambda:x\in V
    \longmapsto\pi_W(\lambda x)\in\Fm/W,$
    where $\pi_W\colon\Fm\longrightarrow\Fm/W$ is the canonical quotient map. Let
    $$\mathcal{A}_{V,W}=\{A_\lambda:\lambda\in\Fm\}\subseteq
    \mathrm{Hom}_{\Fq}(V,\Fm/W),$$
where $\mathrm{Hom}_{\Fq}(V,\Fm/W)$ denotes the $\Fq$-vector
space of all $\Fq$-linear maps from $V$ to $\Fm/W$.
Since $\dim_{\Fq}(V)=d$ and $\dim_{\Fq}(\Fm/W)=m-\delta$,
choosing $\Fq$-bases of $V$ and $\Fm/W$ identifies $\mathrm{Hom}_{\Fq}(V,\Fm/W)$ with $\Fq^{(m-\delta)\times d}$, preserving ranks. Moreover, since the map $\lambda\longmapsto A_\lambda$ is $\Fq$-linear, $\mathcal{A}_{V,W}$ is an $\Fq$-linear rank-metric code in $\Fq^{(m-\delta)\times d}$. In particular, when $m=n$, the ambient space is the space $\Fq^{d\times d}$ of square matrices. We compute $\rho(\mathcal A_{V,W})$ in the ambient space $\mathrm{Hom}_{\Fq}(V,\Fm/W)$ endowed with the rank metric.

\begin{theorem}\label{thm:assmat}
Let $\mathcal C=\left\langle(v_1,\ldots,v_d,0,\ldots,0)\right\rangle_{\Fm}$
be an $[n,1,d]_{q^m/q}$ code, let $V=\langle v_1,\ldots,v_d\rangle_{\Fq}$,
and let $\delta=n-d$. Assume that $0\leq\delta\leq m$. Then
$$\rho(\mathcal C)=\delta+\max_{\substack{W\subseteq\Fm\\\dim_{\Fq}(W)=\delta}}\rho(\mathcal A_{V,W}).$$
\end{theorem}

\begin{proof}
If $\delta=m$, Corollary \ref{cor:uppb} gives
$\rho(\mathcal C)=m$. On the other hand, the only possible choice is
$W=\Fm$, for which $\Fm/W=\{0\}$, hence $\rho(\mathcal A_{V,W})=0$. Thus, the statement holds. Assume that $\delta<m$.
By the construction in the proof of Proposition \ref{prop:caractd}, proving $(i)\implies(ii)$, every covering lift $\widetilde U\in\mathcal L_\delta(V)$ can be written as
$$\widetilde U=U_{f,W}=\{(x,f(x)+w):x\in V,\ w\in W\},$$
for some $W\subseteq\Fm$ of $\Fq$-dimension $\delta$ and an $\Fq$-linear map $f\colon V\to\Fm$. Let $\widetilde f=\pi_W\circ f$ and let $P_\lambda=\langle(1,\lambda)\rangle_{\Fm}$, $\lambda\in\Fm$, which are all the points different from $P$ and 
$$\wt_{U_{f,W}}(P_\lambda)=\dim_{\Fq}(\{x\in V:f(x)-\lambda x\in W\})=d-\dim_{\Fq}(\mathrm{Im}(\pi_W\circ(f-\lambda\mathrm{id}_V)))=d-\dim_{\Fq}(\mathrm{Im}(\widetilde f-A_\lambda)).$$
Hence,
$$\mu_P(U_{f,W})=\max_{\lambda\in\Fm}(d-\dim_{\Fq}(\mathrm{Im}(\widetilde f-A_\lambda)))=d-\mathrm d_{\mathrm{rk}}
       \bigl(\widetilde f,\mathcal A_{V,W}\bigr),$$
and we can conclude that
\begin{equation}\label{eq:assmat}
    n-\mu_P(U_{f,W})=\delta+\mathrm d_{\mathrm{rk}}
       \bigl(\widetilde f,\mathcal A_{V,W}\bigr).
\end{equation}
Applying \eqref{eq:assmat} and Corollary \ref{cor:k1} we have
$$\rho(\mathcal{C})=n-\min_{\widetilde U\in\mathcal{L}_\delta(V)}\mu_P(\widetilde U)=\max_{\widetilde U\in\mathcal{L}_\delta(V)}\left(n-\mu_P(\widetilde U)\right)=\max_{f,W\;:\; U_{f,W}\in\mathcal{L}_{\delta}(V)}\left(\delta+\mathrm d_{\mathrm{rk}}
       \left(\widetilde f,\mathcal A_{V,W}\right)\right).$$
Observe that
$$\max_{f,W\;:\; U_{f,W}\in\mathcal{L}_{\delta}(V)}\left(\mathrm d_{\mathrm{rk}}\left(\widetilde f,\mathcal A_{V,W}\right)\right)=\max_{\substack{W\subseteq\Fm\\ \dim_{\Fq}(W)=\delta}}
\ \max_{\substack{f\in\mathrm{Hom}_{\Fq}(V,\Fm)\\
                   U_{f,W}\in\mathcal L_\delta(V)}}
\mathrm d_{\mathrm{rk}}\bigl(\widetilde f,\mathcal A_{V,W}\bigr).$$
For each $W$, consider the surjective map
$$\psi_W:f\in\mathrm{Hom}_{\Fq}(V,\Fm)\longrightarrow\widetilde f=\pi_W\circ f\in\mathrm{Hom}_{\Fq}(V,\Fm/W).$$
If $\delta>0$, every $f\in\mathrm{Hom}_{\Fq}(V,\Fm)$ is admissible. Hence, by the surjectivity of $\psi_W$,
$$\max_{f\in\mathrm{Hom}_{\Fq}(V,\Fm)}\mathrm d_{\mathrm{rk}}\bigl(\widetilde f,\mathcal A_{V,W}\bigr)=\max_{\widetilde f\in\mathrm{Hom}_{\Fq}(V,\Fm/W)}\mathrm d_{\mathrm{rk}}\bigl(\widetilde f,\mathcal A_{V,W}\bigr)=\rho(\mathcal A_{V,W}).$$
Taking the maximum over all $\delta$-dimensional subspaces $W\subseteq\Fm$ gives the statement when $\delta>0$.
Suppose now that $\delta=0$. Then $W=\{0\}$,
$\widetilde f=f$, and $U_{f,0}=\{(x,f(x)):x\in V\}$.
This subspace is a $0$-covering lift of $V$ if and only if it
spans $\Fm^2$ over $\Fm$. Since its projection onto $V$ is nonzero,
it fails to span $\Fm^2$ precisely when it is contained in some
$P_\lambda=\langle(1,\lambda)\rangle_{\Fm}$, or equivalently when
$f(x)=\lambda x$ for every $x\in V$. Thus, the maps giving
$0$-covering lifts are exactly those outside $\mathcal A_{V,0}$. The maps in $\mathcal A_{V,0}$ clearly have rank distance zero from
$\mathcal A_{V,0}$. Moreover, since $d=n>1$,
$$\dim_{\Fq}(\mathcal A_{V,0})=m<dm=\dim_{\Fq}(\mathrm{Hom}_{\Fq}(V,\Fm)).$$
Consequently, maps outside $\mathcal A_{V,0}$ exist and have
positive distance from it. Excluding the maps in
$\mathcal A_{V,0}$ does not change the maximum:
$$\max_{f\notin\mathcal A_{V,0}}\mathrm d_{\mathrm{rk}}(f,\mathcal A_{V,0})=\max_{f\in\mathrm{Hom}_{\Fq}(V,\Fm)}\mathrm d_{\mathrm{rk}}(f,\mathcal A_{V,0})=\rho(\mathcal A_{V,0}).$$
This completes the proof.

\end{proof}

\begin{corollary}\label{cor:assmat}
With the notation of Theorem \ref{thm:assmat}, assume that
$\delta=n-d<m$. Then $\rho(\mathcal C)=\min\{m,n\}-1$
if and only if there exists an $\Fq$-subspace $W\subseteq\Fm$ of
dimension $n-d$ such that $\rho(\mathcal A_{V,W})=\min\{m,n\}-\delta-1.$
\end{corollary}

\begin{proof}
Fix an $\Fq$-subspace $W\subseteq\Fm$ of dimension $\delta$, then $\dim_{\Fq}(\Fm/W)=m-\delta$, which is positive since by hypothesis $\delta<m$.
We first prove that
$$\rho(\mathcal A_{V,W})
\leq
\min\{d-1,m-\delta-1\}.$$
First, let $B\in\mathrm{Hom}_{\Fq}(V,\Fm/W)$ and fix $x\in V\setminus\{0\}$. Let $y\in\Fm$ such that $\pi_W(y)=B(x)$ and take $\lambda=y/x$. Then
$$A_\lambda(x)=\pi_W(\lambda x)=\pi_W(y)=B(x),$$
hence $x\in\ker(B-A_\lambda)$ and
$\mathrm{d}_{\mathrm{rk}}(B,\mathcal A_{V,W})\leq\mathrm{rk}(B-A_\lambda)\leq d-1.$
On the other hand, fix a basis of $\Fm/W$ and recall that $V=\langle v_1,\ldots,v_d\rangle_{\Fq}$. For every $\lambda\in\Fm$, let $R_\lambda\in\Fq^d$ denote the first row of
the matrix representing $A_\lambda$. We claim that
$$\{R_\lambda:\lambda\in\Fm\}=\Fq^d.$$
Otherwise, since the map $\lambda\mapsto R_\lambda$ is
$\Fq$-linear, there would exist
$c=(c_1,\ldots,c_d)\in\Fq^d\setminus\{0\}$ such that $R_\lambda c^{\top}=0$ for every $\lambda\in\Fm$. If we consider $x=c_1v_1+\cdots+c_dv_d$, we have $x\neq0$, and $R_\lambda c^{\top}$ is the first coordinate of $A_\lambda(x)=\pi_W(\lambda x)$.
As $\lambda$ ranges over $\Fm$, the elements $\lambda x$ range
over the whole field $\Fm$, and hence the elements
$\pi_W(\lambda x)$ range over the whole quotient $\Fm/W$.
This contradicts the fact that their first coordinate is always
zero.

Thus, for every $B$, there exists $\lambda\in\Fm$ such that the
first row of $A_\lambda$ coincides with the first row of $B$.
Hence, the matrix representing $B-A_\lambda$ has a zero row,
and $\mathrm{d}_{\mathrm{rk}}(B,\mathcal A_{V,W})\leq m-\delta-1$.
It follows that $\rho(\mathcal A_{V,W})\leq\min\{d-1,m-\delta-1\}$. Since $d=n-\delta$, we have
$$\rho(\mathcal A_{V,W})\leq\min\{d-1,m-\delta-1\}=\min\{n-\delta-1,m-\delta-1\}=\min\{m,n\}-\delta-1.$$
By Theorem \ref{thm:assmat} we have
$$\rho(\mathcal C)=\delta+\max_{\substack{W\subseteq\Fm\\\dim_{\Fq}(W)=\delta}}\rho(\mathcal A_{V,W})\leq\min\{m,n\}-1.$$
Equality holds if and only if the maximum over $W$ is equal to
$\min\{m,n\}-\delta-1$, which is equivalent to the existence of an
$\Fq$-subspace $W\subseteq\Fm$ of dimension $n-d$ such that
$\rho(\mathcal A_{V,W})=\min\{m,n\}-\delta-1$.

\end{proof}

\begin{remark}
Fix an $\Fq$-subspace $W\subseteq\Fm$ of dimension $\delta$ and let $B\in\mathrm{Hom}_{\Fq}(V,\Fm/W)$. Choose an $\Fq$-linear map $f\colon V\longrightarrow\Fm$ such that $B=\pi_W\circ f$. Whenever the associated subspace $U_{f,W}$ is a covering lift, the proof of Theorem \ref{thm:assmat} gives $$n-\mu_P(U_{f,W})=\delta+\mathrm{d}_{\mathrm{rk}}(B,\mathcal A_{V,W}).$$ Consequently, we can characterize the deep holes of $\mathcal A_{V,W}$, namely the elements $B\in\mathrm{Hom}_{\Fq}(V,\Fm/W)$ such that $\drk(B,\mathcal A_{V,W})=\rho(\mathcal A_{V,W})$. Indeed, $B$ is a deep hole of $\mathcal A_{V,W}$ if and only if $U_{f,W}$ maximizes $n-\mu_P(U_{f,W})$ among the covering lifts having kernel $\{0\}\times W$.

In particular, if $W$ satisfies the equality in Corollary \ref{cor:assmat}, then an $\Fq$-linear map $f\colon V\longrightarrow\Fm$ satisfies \eqref{cond:d} if and only if $\pi_W\circ f$ is a deep hole of $\mathcal A_{V,W}$.
\end{remark}

\subsection{Boundary cases}

The following theorem determines the covering radius for some boundary values of the minimum distance, using known constructions of clubs.

\begin{theorem}\label{thm:boundary}
Let $\mathcal{C}$ be an $[n,1,d]_{q^m/q}$ code, let $\delta=n-d$ and suppose that $d\in\{1,2,\min\{m,n\}-1,\min\{m,n\}\}.$
Then
$$
\rho(\mathcal{C})=
\begin{cases}
m, & \text{if } \delta\geq m,\\
\min\{m,n\}-1, & \text{if } \delta<m.
\end{cases}
$$
\end{theorem}

\begin{proof}
For a $1$-dimensional code, as observed in Remark \ref{rem:n'd} the effective length coincides with its minimum distance $d$. Therefore, if $\delta\geq m$, the equality $\rho(\mathcal{C})=m$ follows from Proposition \ref{prop:uppb}. Hence, from now on, suppose that $\delta<m$ and let $V=\langle v_1,\ldots,v_d\rangle_{\Fq}$

\medskip

\underline{\textbf{Case $d=\min\{m,n\}$.}} Let $m\geq n$. An $[n,1,n]_{q^m/q}$ code has the form $\mathcal{C}=\langle (v_1,\ldots,v_n)\rangle_{\Fm}$, with $\dim_{\Fq}(\langle v_1,\ldots,v_n\rangle_{\Fq})=n$. We want to find $\widetilde U\in\mathcal{L}_\delta(V)$ being a scattered subspace. By taking $\alpha=(\alpha_1,\ldots,\alpha_n)=(v_1^q,\ldots,v_n^q)$ we obtain the $q$-system
\begin{equation*}
    \widetilde U=\langle(v_1,v_1^q),\ldots,(v_n,v_n^q)\rangle_{\Fq},
\end{equation*}
which is a scattered subspace of pseudoregulus type. Hence, by Corollary \ref{cor:corrclub} we get $\rho(\mathcal{C})= n-1=\min\{m,n\}-1$.

Let $m<n<2m$ and $\mathcal
C=\langle(v_1,\ldots,v_m,0,\ldots,0)\rangle_{\Fm}$ be an $[n,1,m]_{q^m/q}$ code. Then we can consider an $[m,1,m]_{q^m/q}$ code $\mathcal
C'=\langle(v_1,\ldots,v_m)\rangle_{\Fm}$ and from the previous case we know that there exists $\widetilde U\in\mathcal{L}_{m-d}(V)$ which is a scattered subspace of pseudoregulus type of rank $m$. Hence, by Corollary \ref{cor:n>m} we can conclude that $\rho(\mathcal{C})= m-1=\min\{m,n\}-1$.

\medskip
\underline{\textbf{Case $d=\min\{m,n\}-1$.}} 
Let $m\ge n$ and let $\mathcal{C}=\langle (v_1,\ldots,v_{n-1},0)\rangle_{\Fm}$ be an $[n,1,n-1]_{q^m/q}$ code, with $v_1,\ldots,v_{n-1}$ linearly independent over $\Fq$. In order to prove that $\rho(\mathcal C)\ge n-1$, we look for a $\widetilde U\in\mathcal{L}_\delta(V)$ which is a scattered subspace of rank $n$ containing $P=\langle(0,1)\rangle_{\Fm}$. Let $\theta$ denote the Frobenius automorphism of $\Fm$, i.e. $\theta(x)=x^q$ for every $x\in\Fm$ and consider $$f:=\theta-\operatorname{id}:x\in\Fm\longrightarrow x^q-x\in\Fm,$$ which is an $\Fq$-linear map with $\mathrm{ker}(f)=\Fq$ and $\mathrm{Im}(f)=\ker(\Trm)$, where $\Trm$ denotes the trace of $\Fm$ over $\Fq$, i.e. $\Trm(z)=z+z^q+\ldots+z^{q^{m-1}}$, $z\in\Fm$. Consider the orthogonal complement of $V$ with respect to the
nondegenerate trace bilinear form, i.e. $\langle x,y\rangle=\Trm(xy)$, $x,y\in\Fm$,  $$V^\perp=\{\mu\in\Fm : \Trm(\mu v)=0, \;\;\forall\; v\in V\}$$
and observe that $\lambda\in V^\perp$ if and only if $\lambda V\subseteq\ker(\Trm)=\mathrm{Im}(f)$. Since $\dim_{\Fq}(V^\perp)=m-n+1\ge 1$, there exists $\lambda\in V^\perp\setminus\{0\}$. Define $\widetilde V=\lambda V$, choose an $\Fq$-basis $\widetilde v_1,\ldots,\widetilde v_{n-1}$ of $\widetilde V$, and let $\widetilde{\mathcal C}=\langle(\widetilde v_1,\ldots,\widetilde v_{n-1},0)\rangle_{\Fm}.$ Since $\widetilde V=\lambda V$, by Lemma \ref{lem:equivalence1d} the codes
$\mathcal C$ and $\widetilde{\mathcal C}$ are equivalent. In particular, they have the same covering radius. Thus, after replacing $\mathcal C$ by an equivalent code and relabelling, we can assume that $V\subseteq\mathrm{Im}(f)$. Hence, for every $i\in\{1,\ldots,n-1\}$, there exists $x_i\in\Fm$ such that $v_i=f(x_i)$ and $0=v_n=f(x_n)$, with $x_n\in\Fq^\ast$, thus $\mathcal{C}=\langle(f(x_1),\ldots,f(x_n))\rangle_{\Fm}$. Moreover, $x_1,\ldots,x_n$ are linearly independent over $\Fq$. Indeed, if
$$y=\sum_{i=1}^{n-1}a_i x_i\in\langle x_1,\ldots,x_{n-1}\rangle_{\Fq}\cap\Fq,$$ then
$$0=f(y)=\sum_{i=1}^{n-1}a_i f(x_i)=\sum_{i=1}^{n-1}a_i v_i.$$
The $\Fq$-linear independence of $v_1,\ldots,v_{n-1}$ implies that $a_1=\cdots=a_{n-1}=0$, and hence $y=0$. Therefore,
$\langle x_1,\ldots,x_{n-1}\rangle_{\Fq}\cap\Fq=\{0\},$ and, since $x_n\in\Fq^\ast$, the elements $x_1,\ldots,x_n$ are
$\Fq$-linearly independent. Since $\mathcal{C}=\langle(f(x_1),\ldots,f(x_n))\rangle_{\Fm}$ with $x_1,\ldots,x_n$ linearly independent over $\Fq$, in order to obtain a scattered subspace it is sufficient to consider 
    \begin{equation*}
        \widetilde U=\langle (f(x_1),x_1),\ldots,(f(x_n),x_n)\rangle_{\Fq}.
    \end{equation*}
     This is a scattered subspace contained in one of pseudoregulus type of rank $n$, since it is equivalent to $\langle (x_1,x_1^q),\ldots,(x_n,x_n^q)\rangle_{\Fq}$. Hence, by Corollary \ref{cor:corrclub} we get $\rho(\mathcal{C})= n-1=\min\{m,n\}-1$.

Assume now that $m<n<2m-1$ and let $\mathcal
C=\langle(v_1,\ldots,v_{m-1},0,\ldots,0)\rangle_{\Fm}$ be an $[n,1,m-1]_{q^m/q}$ code. Then we can consider an $[m,1,m-1]_{q^m/q}$ code $\mathcal
C'=\langle(v_1,\ldots,v_{m-1},0)\rangle_{\Fm}$ and we know that in this case there exists $\widetilde U\in\mathcal{L}_{m-d}(V)$ which is a scattered subspace of rank $m$ containing $P=\langle(0,1)\rangle_{\Fm}$. Hence, by Corollary \ref{cor:n>m} we can conclude that $\rho(\mathcal{C})= m-1=\min\{m,n\}-1$.

\medskip

\underline{\textbf{Case $d=1$.}}
Let $m\geq n$ be positive integers. If we consider an $[n,1,1]_{q^m/q}$ code $\mathcal{C}$, up to equivalence we can write $\mathcal{C}=\langle(1,0,\ldots,0)\rangle_{\Fm}$ and we want to find $\widetilde U\in\mathcal{L}_\delta(V)$ which is an $(n-1)$-club. If we consider any $\Fq$-basis of $\Fm$, for instance let $\omega\in\Fm$ of degree $m$ over $\Fq$ and consider $\mathcal{B}=\{1,\omega,\ldots,\omega^{m-1}\}$ as $\Fq$-basis of $\Fm$. Take the $[n,2]_{q^m/q}$ system

\begin{equation*}
    \widetilde U=\langle(1,0),(0,1),(0,\omega),\ldots,(0,\omega^{n-2})\rangle_{\Fq},
\end{equation*}
this is an $(n-1)$-club of rank $n$ and $\pi(\widetilde U)=V$. With these choices
\begin{align*}
    &\wt_{\widetilde U}(\langle(0,1)\rangle_{\Fm})=n-1\\
    &\wt_{\widetilde U}(\langle(1,\beta)\rangle_{\Fm})\leq1, \quad \forall\beta\in\Fm.
\end{align*}
Hence, by Corollary \ref{cor:corrclub} we can conclude that $\rho(\mathcal{C})= n-1$.

\medskip

\underline{\textbf{Case $d=2$.}} Since $\delta=n-2<m$, we have $n\leq m+1$. We first consider the case $n=m$. Up to equivalence, we can write $\mathcal{C}=\langle(1,\beta,0,\ldots,0)\rangle_{\Fm}$,
$\beta\in\Fm\setminus\Fq.$

Our aim is to find $\widetilde U\in\mathcal{L}_\delta(V)$ which is a $\delta$-club of rank $m$. In coordinates, we need to find $\alpha=(\alpha_1,\ldots,\alpha_m)\in\Fm^m$ such that $\widetilde U=\langle(1,\alpha_1),(\beta,\alpha_2),(0,\alpha_3),\ldots,(0,\alpha_m)\rangle_{\Fq}$ is a $\delta$-club of rank $m$. Since the case $m=n=2$ was already covered, we can assume that $m\ge3$. By Theorem \ref{thm:clubs_cod2}, it is enough to find an $(m-2)$-dimensional $\Fq$-subspace $T$ of $\Fm$ containing $1$ and an element $\gamma\in\Fm\setminus(T+\beta T)$, where $$T+\beta T=\{t_1+\beta t_2 \;:\; t_1,t_2\in T\}.$$
Indeed, if $t_1,\ldots,t_{m-2}$ is an $\Fq$-basis of $T$, then
$$U_{\alpha}=\langle(1,1),(\beta,\gamma),(0,t_1),\ldots,(0,t_{m-2})\rangle_{\Fq}$$ is an $(m-2)$-club of rank $m$. We put $\alpha_1=1$ and we need to find $\alpha_3,\ldots,\alpha_m\in\Fm$ such that $T=\langle\alpha_3,\ldots,\alpha_m\rangle_{\Fq}$ is a subspace of $\Fq$-dimension $m-2$ and $\gamma=\alpha_2\notin T+\beta T$.  In order to obtain this result, we need to consider three different cases depending on which subfield contains $\beta$.
\begin{itemize}
    \item If $\beta$ is not contained in any proper subfield, i.e. $\Fq(\beta)=\Fm$, then we can take $T=\langle1,\beta,\ldots,\beta^{m-3}\rangle_{\Fq}$. Since $T+\beta T=\langle1,\beta,\ldots,\beta^{m-3},\beta^{m-2}\rangle_{\Fq}$, we can choose $\gamma=\beta^{m-1}$. In this case the $q$-system obtained is
    \begin{equation*}
        U_{\alpha}=\langle(1,1),(\beta,\beta^{m-1}),(0,1),(0,\beta),\ldots,(0,\beta^{m-3})\rangle_{\Fq},
    \end{equation*}
    and it is an $(m-2)$-club.
    \item Suppose that the smallest proper subfield containing $\beta$ is $\mathbb{F}_{q^t}$ with $t\geq3$, i.e. $\mathbb{F}_{q}(\beta)=\mathbb{F}_{q^t}$. We can take an $\mathbb{F}_{q^t}$-basis of $\Fm$, namely $\{1,\xi,\ldots,\xi^{r-1}\}$, where $r:=\frac{m}{t}$. We can take $T=\xi\mathbb{F}_{q^t}+\ldots+\xi^{r-1}\mathbb{F}_{q^t}+\langle1,\beta,\ldots,\beta^{t-3}\rangle_{\Fq}$, whose $\Fq$-dimension is  $\dim_{\Fq}(T)=t(r-1)+t-2=m-2$. Since $T+\beta T=\xi\mathbb{F}_{q^t}+\ldots+\xi^{r-1}\mathbb{F}_{q^t}+\langle1,\beta,\ldots,\beta^{t-2}\rangle_{\Fq}$, we can choose $\gamma=\beta^{t-1}$. In this second case, the $q$-system obtained is 
    \begin{equation*}
        \widetilde U=\langle(1,1),(\beta,\beta^{t-1}),(0,1),(0,\beta),\ldots,(0,\beta^{t-3}),(0,\xi),(0,\beta\xi),\ldots,(0,\beta^{t-1}\xi),\ldots,(0,\beta^{t-1}\xi^{r-1})\rangle_{\Fq},
    \end{equation*}
    which is an $(m-2)$-club.
   \item If $m$ is even, it can happen that $\Fq(\beta)=\mathbb{F}_{q^2}$. In this case we can consider an $\mathbb{F}_{q^2}$-basis of $\Fm$, namely $\{1,\xi,\ldots,\xi^{r-1}\}$ with $r:=\frac{m}{2}$, and take the $\Fq$-subspace $T=\langle1,\xi,\ldots,\xi^{r-2},\beta,\ldots,\beta\xi^{r-2}\rangle_{\Fq}$, which has $\Fq$-dimension $m-2$. In this case $\beta T=T$, $T$ is an $\mathbb{F}_{q^2}$-subspace and we can consider $\gamma=\xi^{r-1}$. Doing these choices we obtain the following $q$-system
   \begin{equation*}
       \widetilde U=\langle(1,1),(\beta,\xi^{r-1}),(0,1),(0,\xi),\ldots,(0,\xi^{r-2}),(0,\beta),\ldots,(0,\beta\xi^{r-2})\rangle_{\Fq}.
   \end{equation*}
\end{itemize}
Since it is always possible to find an $(m-2)$-covering lift of $V$ which is an $(m-2)$-club, by Corollary \ref{cor:corrclub} we can conclude that $\rho(\mathcal{C})= m-1$.

 Let $m>n$, the case $n=2$ has already been covered by the first case, so assume that $n\geq3$. Let $\mathcal{C}=\langle(1,\beta,0,\ldots,0)\rangle_{\Fm}$  be an $[n,1,2]_{q^m/q}$ code. If we consider $\mathcal{C}'=\langle (1,\beta,0,\ldots,0)\rangle_{\Fm}$ an $[m,1,2]_{q^m/q}$ code, hence there exist an $(m-2)$-subspace $T'=\langle 1,t'_2,\ldots,t'_{m-2}\rangle_{\Fq}$ and $\gamma \in \Fm\setminus (T'+\beta T')$ such that
$$\widetilde U'=\langle(1,1),(\beta,\gamma),(0,1)\ldots,(0,t'_{m-2})\rangle_{\Fq},$$
is an $(m-2)$-covering lift which is an $(m-2)$-club. Consider $T=\langle 1,t_2,\ldots,t_{n-2}\rangle_{\Fq}\subseteq T'$ to be an $(n-2)$-subspace, and define
$$\widetilde U=\langle(1,1),(\beta,\gamma),(0,1),(0,t_2),\ldots(0,t_{n-2})\rangle_{\Fq}.$$
Since for every $Q\neq P$, $\widetilde U\cap Q\subseteq \widetilde U'\cap Q$, then 
$$\wt_{\widetilde U}(Q)=\dim_{\Fq}(\widetilde U\cap Q)\leq\dim_{\Fq}(\widetilde U'\cap Q)=\wt_{\widetilde U'}(Q),$$
while $\wt_{\widetilde U}(P)=\dim_{\Fq}(T)=n-2$. Hence, $\widetilde U\in\mathcal{L}_{\delta}(V)$ and it is an $(n-2)$-club of rank $n$, thus by Corollary \ref{cor:corrclub} we get $\rho(\mathcal{C})= n-1$. 

 \medskip
 Let $n=m+1$ and let $\mathcal
C=\langle(1,\beta,0,\ldots,0)\rangle_{\Fm}$, with $\beta\in\Fm\setminus\Fq$, be an $[n,1,2]_{q^m/q}$ code. Then we can consider an $[m,1,2]_{q^m/q}$ code $\mathcal
C'=\langle(1,\beta,0,\ldots,0)\rangle_{\Fm}$ and we know that in this case there exists $\widetilde U\in\mathcal{L}_{m-2}(V)$ which is an $(m-2)$-club of rank $m$. Hence, by Corollary \ref{cor:n>m} we can conclude that $\rho(\mathcal{C})= m-1$.

\end{proof}

By Theorem \ref{thm:boundary} and  Corollary \ref{cor:taglik1} we can derive another result on the covering radius of 1-dimensional rank-metric codes.

\begin{proposition} \label{prop:taglik1r12}
    Let $n\ge m$ and let $s\in\{0,\ldots,n-m\}$. Then
    \begin{enumerate}[label=(\arabic*)]
        \item Every $[n,1,s+1]_{q^m/q}$ code has covering radius  at least $m-1$.
        \item Every $[n,1,s+2]_{q^m/q}$ code has covering radius  at least $m-1$.
    \end{enumerate}
\end{proposition}

\subsection{Special cases} \label{subsec:special_cases}
We conclude this section on the covering radius of $1$-dimensional codes presenting an overview on some cases involving small values for the degree of the extension $m$. We focus on the cases $m\in\{3,4,5,6\}$ using the results appearing in Theorem \ref{thm:boundary}, Proposition \ref{prop:taglik1r12} and Proposition \ref{prop:boundGY}.
To avoid redundancy, Tables \ref{tab:m=3}--\ref{tab:m=6} display the range $n\geq m$ except for Table \ref{tab:m=6}, where the column $n=5$ is included. Indeed, every omitted case with $n<m$ is already determined by Theorem \ref{thm:boundary}, while $n=5$, $m=6$, $d=3$ is the unique admissible triple not covered by the boundary cases and requires attention.

\begin{proposition}\label{prop:special}
    Let $m\in\{3,4,5\}$ and let $\mathcal{C}$ be an $[n,1,d]_{q^m/q}$ code. If $n\ge m+d$, then $\rho(\mathcal{C})=m$. If $n<m+d$, then $\rho(\mathcal{C})=\min\{m,n\}-1$ except when $(m,n,d)=(5,5,3)$, in which case $\rho(\mathcal{C})=3$.
\end{proposition}
\begin{proof}
    For $m\in\{3,4\}$ and every value of $d$ and $n$ the statement follows directly from Theorem \ref{thm:boundary}. For $m=5$, every case different from $(m,n,d)=(5,5,3)$ can be retrieved by Theorem \ref{thm:boundary} and Proposition \ref{prop:taglik1r12}. For $(m,n,d)=(5,5,3)$, by Theorem \ref{thm:no2club} there is no $2$-club of rank $5$ in $\mathbb{F}_{q^5}^2$, hence for $d=3$ and $n=5$ we have that $\rho(\mathcal{C})\leq3$. Moreover, by the lower bound of Proposition \ref{prop:boundGY}, we can conclude that $\rho(\mathcal{C})=3$.
    
\end{proof} 

The results in Proposition \ref{prop:special} are summarized in the following tables.

\begin{table}[H]
\centering
\begin{tabular}{c|c|c|c|c}
 & $n=3$ & $n=4$ & $n=5$ & $n\geq6$ \\
\hline
$d=1$ & $\rho(\mathcal C)=2$ & $\rho(\mathcal C)=3$ & $\rho(\mathcal C)=3$ & $\rho(\mathcal C)=3$ \\
\hline
$d=2$ & $\rho(\mathcal C)=2$ & $\rho(\mathcal C)=2$ & $\rho(\mathcal C)=3$ & $\rho(\mathcal C)=3$\\
\hline
$d=3$ & $\rho(\mathcal C)=2$ & $\rho(\mathcal C)=2$ & $\rho(\mathcal C)=2$ & $\rho(\mathcal C)=3$\\
\end{tabular}
\caption{$m=3$}
\label{tab:m=3}
\end{table}

\begin{table}[H]
\centering
\begin{tabular}{c|c|c|c|c|c}
 & $n=4$ & $n=5$ & $n=6$ & $n=7$ & $n\geq8$\\
\hline
$d=1$ & $\rho(\mathcal C)=3$ & $\rho(\mathcal C)=4$ & $\rho(\mathcal C)=4$ & $\rho(\mathcal C)=4$ & $\rho(\mathcal C)=4$\\
\hline
$d=2$ & $\rho(\mathcal C)=3$ & $\rho(\mathcal C)=3$ & $\rho(\mathcal C)=4$ & $\rho(\mathcal C)=4$ & $\rho(\mathcal C)=4$\\
\hline
$d=3$ & $\rho(\mathcal C)=3$ & $\rho(\mathcal C)=3$ & $\rho(\mathcal C)=3$ & $\rho(\mathcal C)=4$ & $\rho(\mathcal C)=4$\\
\hline
$d=4$ & $\rho(\mathcal C)=3$ & $\rho(\mathcal C)=3$ & $\rho(\mathcal C)=3$ & $\rho(\mathcal C)=3$ & $\rho(\mathcal C)=4$\\
\end{tabular}
\caption{$m=4$}
\label{tab:m=4}
\end{table}

\begin{table}[H]
\centering
\begin{tabular}{c|c|c|c|c|c|c}
 & $n=5$ & $n=6$ & $n=7$ & $n=8$ & $n=9$ & $n\geq10$\\
\hline
$d=1$ & $\rho(\mathcal C)=4$ & $\rho(\mathcal C)=5$ & $\rho(\mathcal C)=5$ & $\rho(\mathcal C)=5$ & $\rho(\mathcal C)=5$ & $\rho(\mathcal C)=5$\\
\hline
$d=2$ & $\rho(\mathcal C)=4$ & $\rho(\mathcal C)=4$ & $\rho(\mathcal C)=5$ & $\rho(\mathcal C)=5$ & $\rho(\mathcal C)=5$ & $\rho(\mathcal C)=5$\\
\hline
$d=3$ & $\rho(\mathcal C)=3$ & $\rho(\mathcal C)=4$ & $\rho(\mathcal C)=4$ & $\rho(\mathcal C)=5$ & $\rho(\mathcal C)=5$ & $\rho(\mathcal C)=5$\\
\hline
$d=4$ & $\rho(\mathcal C)=4$ & $\rho(\mathcal C)=4$ & $\rho(\mathcal C)=4$ & $\rho(\mathcal C)=4$ & $\rho(\mathcal C)=5$ & $\rho(\mathcal C)=5$\\
\hline
$d=5$ & $\rho(\mathcal C)=4$ & $\rho(\mathcal C)=4$ & $\rho(\mathcal C)=4$ & $\rho(\mathcal C)=4$ & $\rho(\mathcal C)=4$ & $\rho(\mathcal C)=5$\\

\end{tabular}
\caption{$m=5$}
\label{tab:m=5}
\end{table}

 In particular, it is worth mentioning that every $[4,1]_{q^4/q}$ code has covering radius $3$. Hence, using this new geometric perspective, we can recover the result obtained by Zullo in \cite{Z2024}, who used the correspondence with rank-$\rho$-saturating systems.

\begin{example}
    For $m=5$, the cases $d=3$ and $n=6$ or $n=7$ provide two examples showing that the sufficient condition in Corollary \ref{cor:n>m} is not necessary. For instance, consider a $[6,1,3]_{q^5/q}$ code $\mathcal{C}=\langle(v_1,v_2,v_3,0,0,0)\rangle_{\mathbb{F}_{q^5}}$. By Proposition \ref{prop:taglik1r12} and Corollary \ref{cor:uppb} this code has covering radius $4$, but it does not satisfy the sufficient hypothesis in Corollary \ref{cor:n>m}. Since there is no $2$-club of rank $5$ in $\mathbb{F}_{q^5}^2$, there is no $2$-covering lift of $V=\langle v_1,v_2,v_3\rangle_{\Fq}$ which is a $2$-club of rank $5$ in $\mathbb{F}_{q^5}^2$.
\end{example}

For $m=6$ we have four cases that are still open using our results, while by Proposition \ref{prop:boundGY} we can at least say that the covering radius of $[5,1,3]_{q^6/q}$ codes can be either $3$ or $4$, while for $[6,1,3]_{q^6/q}$ and $[6,1,4]_{q^6/q}$ codes we can conclude that $4\le\rho(\mathcal{C})\le 5$. As shown in the proof of Proposition \ref{prop:tagli}, $\rho(\mathcal{C})\ge\rho(\Pi(\mathcal{C},I))$, for every $I\subseteq[n]$. Hence, if $\mathcal{C}$ is a $[7,1,4]_{q^6/q}$ code and we puncture it in a zero coordinate, we get $\rho(\mathcal{C})\ge 4$, hence also in this case $4\le\rho(\mathcal{C})\le 5$. 

Moreover, for $d=3$ and $n=6$ there is a construction of $3$-clubs, that can be derived from Proposition \ref{prop:3club}.

\begin{proposition} \label{prop:consm6}
    Let $\mathcal{C}=\langle(v_1,v_2,v_3,0,0,0)\rangle_{\F_{q^6}}$ be a $[6,1,3]_{q^6/q}$ code. If $V=\langle v_1,v_2,v_3\rangle_{\Fq}=\lambda\F_{q^3}$ with $\lambda\in\F_{q^6}^*$, then $\rho(\mathcal{C})=5$.
\end{proposition}
\begin{proof} Let $u_i=\lambda^{-1}v_i$, then $\mathcal{C}=\langle(u_1,u_2,u_3,0,0,0)\rangle_{\Fm}$, and by Corollary \ref{cor:corrclub} $\rho(\mathcal{C})=5$ if and only if there exists $\widetilde U\in\mathcal{L}_3(V)$ which is a $3$-club with rank $6$. Replacing the generator by $\lambda^{-1}(v_1,v_2,v_3,0,0,0)$, we can assume $V=\F_{q^3}$. Since $m=6=tr$ with $t=3$ and $r=2$, we can use the construction in Proposition \ref{prop:3club}. Choose $f(x)=x^q$, $a=0$, $b=1$ and observe that $\Gamma_f$ is scattered and $f$ is invertible since it is a field automorphism. Let $\omega\in\F_{q^6}$ such that $\{1,\omega\}$ is an $\F_{q^3}$-basis of $\F_{q^6}$, then
\begin{equation*}
    U_{0,1}=\left\{(x^q,x+y\omega) : x,y\in\F_{q^3}\right\}
\end{equation*}
is a $3$-club of rank $6$. The code $\mathcal{C}$ can be extended with this construction: let $x_i=u_i^{q^2}$, then $u_i=x_i^q$ and fixed $\{1,\xi,\xi^2\}$ an $\Fq$-basis of $\F_{q^3}$, we can consider 
\begin{equation*}
    \widetilde U=U_{0,1}=\langle(x_1^q,x_1),(x_2^q,x_2),(x_3^q,x_3),(0,\omega),(0,\xi\omega),(0,\xi^2\omega)\rangle_{\Fq}.
\end{equation*}
Observe that $\pi(\widetilde U)=V$ and it is a $[6,2]_{q^6/q}$ system, hence it is a $3$-covering lift of $V$. Moreover, $\widetilde U$ is a $3$-club with rank $6$ and $\wt_{\widetilde U}(\langle(0,1)\rangle_{\F_{q^6}})=3$. By Corollary \ref{cor:corrclub}, $\rho(\mathcal{C})=5$.
\end{proof}

Proposition \ref{prop:consm6} shows how to construct $[6,1,3]_{q^6/q}$ codes with covering radius $5$. Its converse is not true in general, i.e. if $\mathcal{C}=\langle(v_1,v_2,v_3,0,0,0)\rangle_{\F_{q^6}}$ is a $[6,1,3]_{q^6/q}$ code with covering radius $5$, not necessarily $V=\lambda\F_{q^3}$, with $\lambda\in\F_{q^6}^*$. By MAGMA computations \cite{BCP1997}, we have found an example of such situation. For $q=3$ we have found 28 equivalence classes of $1$-dimensional $[6,1,3]_{3^6/3}$ codes having covering radius $5$, besides the class arising from a scalar multiple of the subfield $\F_{3^3}$.

For instance, let $a\in\F_{3^6}$ be a root of the primitive polynomial $x^6+x+2\in\mathbb{F}_3[x]$, so that $a$ is a primitive element of $\F_{3^6}$ and take the code $\mathcal{C}=\langle(1,a^{282},a^{117},0,0,0)\rangle_{\F_{3^6}}$ with $V=\langle 1,a^{282},a^{117}\rangle_{\F_3}$. By MAGMA computations $\rho(\mathcal{C})=5$ and suppose that $V=\lambda\F_{3^3}$ with $\lambda\in\F_{3^6}^*$. Since $1\in V$, then $\lambda\in\F_{3^3}$, thus $V=\F_{3^3}$. However this is not possible since $a^{282}\notin\F_{3^3}^*=\langle a^{28}\rangle$.

It is worth pointing out that the first counterexample appears for $q=3$. Indeed we have performed an exhaustive computation for the covering radius of $[6,1,3]_{q^6/q}$ codes with $q\in\{2,3\}$. In the first case, the action of $\F_{2^6}^*$ on the set of three-dimensional $\F_2$-subspaces of $\F_{2^6}$ has 23 orbits. Among them, exactly one has covering radius $5$, and it is the class for which the $\F_2$-span of the nonzero coordinates is a scalar multiple of the subfield $\F_{2^3}$. All the remaining $22$ classes have covering radius $4$. 
For $q=3$ the situation is different and in this case the action of $\F_{3^6}^*$ on the set of three-dimensional $\F_3$-subspaces of $\F_{3^6}$ has 94 orbits. Among them, we have found $29$ classes with covering radius $5$, including the one of the counterexample.

\medskip
We have applied the same strategies as for $[6,1,3]_{q^6/q}$ for $q\in\{2,3\}$, also for the other three open cases. For $q\in\{2,3\}$, we have found that every $[6,1,4]_{q^6/q}$ code has covering radius $4$, every $[5,1,3]_{q^6/q}$ code has covering radius $4$ and we have also found that every $[7,1,4]_{q^6/q}$ code has covering radius $5$. All the scripts can be found at \url{https://github.com/AlessiaVallefuoco/Covering-radius-via-clubs}.

\medskip
The following table shows the results for $m=6$ and every value of $q$. 
\begin{table}[H]
\centering
{\small
\setlength{\tabcolsep}{4.5pt}
\begin{tabular}{c|c|c|c|c|c|c|c|c}
& $n=5$ & $n=6$ & $n=7$ & $n=8$ & $n=9$ & $n=10$ & $n=11$ & $n\geq12$\\
\hline
$d=1$ & $\rho(\mathcal C)=4$ & $\rho(\mathcal C)=5$ & $\rho(\mathcal C)=6$ & $\rho(\mathcal C)=6$ & $\rho(\mathcal C)=6$ & $\rho(\mathcal C)=6$ & $\rho(\mathcal C)=6$ & $\rho(\mathcal C)=6$\\
\hline
$d=2$ & $\rho(\mathcal C)=4$ & $\rho(\mathcal C)=5$ & $\rho(\mathcal C)=5$ & $\rho(\mathcal C)=6$ & $\rho(\mathcal C)=6$ & $\rho(\mathcal C)=6$ & $\rho(\mathcal C)=6$ & $\rho(\mathcal C)=6$\\
\hline
$d=3$ & $\rho(\mathcal C)\in\{3^\ast,4\}$ & $\rho(\mathcal C)\in\{4,5\}^{\ast\ast}$ & $\rho(\mathcal C)=5$ & $\rho(\mathcal C)=5$ & $\rho(\mathcal C)=6$ & $\rho(\mathcal C)=6$ & $\rho(\mathcal C)=6$ & $\rho(\mathcal C)=6$\\
\hline
$d=4$ & $\rho(\mathcal C)=4$ & $\rho(\mathcal C)\in\{4,5^\ast\}$ & $\rho(\mathcal C)\in\{4^\ast,5\}$ & $\rho(\mathcal C)=5$ & $\rho(\mathcal C)=5$ & $\rho(\mathcal C)=6$ & $\rho(\mathcal C)=6$ & $\rho(\mathcal C)=6$\\
\hline
$d=5$ & $\rho(\mathcal C)=4$ & $\rho(\mathcal C)=5$ & $\rho(\mathcal C)=5$ & $\rho(\mathcal C)=5$ & $\rho(\mathcal C)=5$ & $\rho(\mathcal C)=5$ & $\rho(\mathcal C)=6$ & $\rho(\mathcal C)=6$\\
\hline
$d=6$ & -- & $\rho(\mathcal C)=5$ & $\rho(\mathcal C)=5$ & $\rho(\mathcal C)=5$ & $\rho(\mathcal C)=5$ & $\rho(\mathcal C)=5$ & $\rho(\mathcal C)=5$ & $\rho(\mathcal C)=6$
\end{tabular}
}
\caption{$m=6$}
\label{tab:m=6}
\vspace{0.5em}
\begin{minipage}{0.95\textwidth}
\footnotesize
\raggedright
$(\ast):\;\,$ MAGMA computations show that, at least for $q\in\{2,3\}$, these values do not occur.\\
$(\ast\ast):$ MAGMA computations show that, at least for $q\in\{2,3\}$, both values occur.
\end{minipage}
\end{table}
It is worth pointing out that, among the cases considered above, the
$[6,1,3]_{q^6/q}$ cases with $q\in\{2,3\}$ are the ones for which our computations exhibit codes with the same length and minimum distance but different covering radii: both $\rho(\mathcal C)=4$ and $\rho(\mathcal C)=5$ occur. Thus, at least for $q\in\{2,3\}$, in these cases the covering radius is not determined by the parameters $n$ and $d$ alone.

\section{Comparison with known results} \label{sec:comparison}
The aim of this section is to compare our approach with some known results on the covering radius of linear rank-metric codes. 

\subsection{Gadouleau and Yan results}
Let us start by analyzing the results on the covering radius of linear rank-metric codes proved by Gadouleau and Yan in \cite{GY2008}.

\begin{proposition} \label{prop:boundGY}
    Let $\mathcal{C}$ be an $[n,k]_{q^m/q}$ code with $m\ge n$. Then
\begin{equation*}
    \left\lfloor\frac{m+n-\sqrt{(m+n)^2-4(mn-mk-2)}}{2}\right\rfloor+1\le\rho(\mathcal{C})\le n-k.
\end{equation*}
In particular if $m=n$, then
     \begin{equation} \label{eq:lowboundGY}
        \left\lfloor m-\sqrt{mk+2}\right\rfloor+1\le\rho(\mathcal{C})\le m-k
    \end{equation}
\end{proposition}
\begin{proof}
By \cite[Proposition 14]{GY2008},
$$n-\rho(\mathcal{C})-\frac{\rho(\mathcal{C})(n-\rho(\mathcal{C}))-\log_q K_q}{m}<k\le n-\rho(\mathcal{C}),$$
where $K_q=\prod_{j\ge1}(1-q^{-j})$. Moreover,
$$K_q\ge K_2
 \approx0.288788>\frac14\ge q^{-2},$$
and hence $\log_q K_q>-2$. This estimate and the first inequality yield $\rho(\mathcal{C})(m+n-\rho(\mathcal{C}))>mn-mk-2$ and solving the corresponding quadratic inequality we obtain
$$\rho(\mathcal{C})\ge
\left\lfloor
\frac{m+n-\sqrt{(m+n)^2-4(mn-mk-2)}}{2}
\right\rfloor+1.$$
The upper bound follows from $k\le n-\rho(\mathcal{C})$. Finally, for $m=n$, the claimed expression follows by simplification.

\end{proof}

The lower bound only works for the case $m\ge n$, therefore it does not say anything for the case $n>m$. In particular we have used the lower bound for $k=1$ in the special cases, when we could not use the characterization provided by the construction of clubs.  However, the upper bound can be stated in its generality, since it works for every value of $m$ and $n$. Indeed, let $\mathcal{C}$ is an $[n,k]_{q^m/q}$ code and let $G$ be a generator matrix of $\mathcal C$. Since $G$ has rank $k$, it contains $k$ linearly independent columns. Hence, for every $x\in\Fm^n$, there exists $c\in\mathcal C$ so that the first $k$ coordinates of $x-c$ are zero. Thus, $x-c$ has at most $n-k$ nonzero coordinates, and hence $\rk(x-c)\leq n-k$. Thus, we can always conclude that $\rho(\mathcal C)\leq\min\{m,n-k\}.$.

\subsection{Byrne and Ravagnani results}
We next consider two upper bounds proved by Byrne and Ravagnani in \cite{BR2017}. Although these results are stated for matrix rank-metric codes, they can be applied to our setting. Indeed, by choosing dual bases of $\Fm$ over $\Fq$ with respect to the trace form, the standard inner product dual of an $\Fm$-linear code is identified, under matrix expansion, with the trace product dual of the corresponding matrix code, preserving rank weights, see \cite[Theorem 21]{R2016}.

The first one is the dual distance bound.

\begin{theorem}[Dual distance bound, {\cite[Corollary 18]{BR2017}}]
    Let $\mathcal{C}$ be an $[n,k]_{q^m/q}$ code, then $$\rho(\mathcal{C})\leq \min\{m,n\}-\drk(\mathcal{C}^\perp)+1.$$
\end{theorem}
In order to apply the dual distance bound in our setting, we recall the following characterization of degenerate $\Fm$-linear codes.

\begin{proposition}[{\cite[Proposition 3.2]{ABNR2022}}]
    Let $\mathcal{C}$ be an $[n,k]_{q^m/q}$ code, then $\drk(\mathcal{C}^\perp)=1$ if and only if $\mathcal{C}$ is degenerate.
\end{proposition}

The dual distance bound, for every degenerate $\Fm$-linear code $\mathcal{C}$ gives the trivial upper bound, which is  $\rho(\mathcal{C})\leq \min\{m,n\}$, while for Proposition \ref{prop:uppb} we have that if $n<m+n'$ where $n'$ is the effective length, then $\rho(\mathcal{C})\leq \min\{m,n\}-1$. The second upper bound proved by Byrne and Ravagnani is the external distance bound.
\begin{theorem}[External distance bound, {\cite[Theorem 24]{BR2017}}]
    Let $m\ge n$ and let $\mathcal{C}$ be an $[n,k]_{q^m/q}$ code. Then $$\rho(\mathcal{C})\leq |\{i\in[n] : \mathcal{A}_i(\mathcal{C}^\perp)>0\}|,$$
    where 
    \begin{equation*}
    \mathcal{A}_i(\mathcal{C}^{\perp}) := |\{c \in \mathcal{C}^\perp : \rk(c) = i\}|.
\end{equation*}
\end{theorem}
The external distance bound for $k=1$ gives the same result as Corollary \ref{cor:uppb} just for $d\in\{1,n\}$, since in the other cases there exist codewords in the dual that achieve every value of weight between 1 and $n$.
\begin{proposition} \label{prop:wdist1dim}
    Let $m\ge n$ and let $\mathcal{C}$ be an $[n,1,d]_{q^m/q}$ code with $1<d<n$. Then $\mathcal{A}_i(\mathcal{C}^\perp)>0$ for every $i\in[n]$.
\end{proposition}
\begin{proof}
    Let $\mathcal{C}=\langle(v_1,\ldots,v_d,0,\ldots,0)\rangle_{\Fm}$ be an $[n,1,d]_{q^m/q}$ code and denote $v=(v_1,\ldots,v_d)\in\Fm^d$, the dual of $\mathcal{C}$ has the form
    \begin{equation*}
        \mathcal{C}^\perp=\{(u,w)\in\Fm^d\times\Fm^{n-d}:(u,w)\cdot (v,0,\ldots,0)=0\}=\{(u,w)\in\Fm^n: u\in \langle v\rangle_{\Fm}^\perp, w\in\Fm^{n-d}\}.
    \end{equation*}
    If we consider $\mathcal{C}$ as isometrically embedded in $\Fm^d$, then $\langle v\rangle_{\Fm}^\perp$ is a $[d,d-1,2]_{q^m/q}$ code. 
    \begin{itemize}
    \item If $i=1$, then we can consider $u=(0,\ldots,0)$ and $w=(w_1,0,\ldots,0)$ with $w_1\in\Fm\setminus\{0\}$.
        \item If $1<i\le d$, the rank-weight distribution of the MRD code
$\langle v\rangle_{\Fm}^\perp$ (see \cite[Proposition 9]{GY2008MW}) implies that there exists $u\in\langle v\rangle_{\Fm}^\perp$ such that $\rk(u)=i$. Moreover, we take $w=(0,\ldots,0)$.
        \item If $d<i\le n$, then there exists $u\in \langle v\rangle_{\Fm}^\perp$ such that $\rk(u)=d$.  Moreover, we can consider  $w=(w_1,\ldots,w_{i-d},0,\ldots,0)\in\Fm^{n-d}$ such that $u_1,\ldots,u_d,w_1,\ldots,w_{i-d}$ are $\Fq$-independent.
        
    \end{itemize}
    Thus, we have just shown that for every $i\in[n]$, there exists $(u,w)\in\mathcal C^\perp$ such that $\rk((u,w))=i$.
    
\end{proof}

By Proposition \ref{prop:wdist1dim}, we can apply the external distance bound for 1-dimensional codes with $1<d<n$ and obtain 
$$\rho(\mathcal{C})\leq |\{i\in[n] : \mathcal{A}_i(\mathcal{C}^\perp)>0\}|=n,$$
while for $d=n$ we do not have codewords of weight 1 and for $d=1$ we do not have codewords of weight $n$, hence $\rho(\mathcal{C})\leq n-1$.

\subsection{Gruica, Ravagnani, Sheekey and Zullo results}

We conclude by comparing Proposition \ref{prop:caractd} and Theorem \ref{thm:assmat} with the geometric interpretation of the covering radius given in \cite[Theorem 6.3]{GRSZ2024}. First, let us briefly recall their setting. Let $\mathcal D\subseteq\Fq^{m\times n}$, with $m\le n$, be a rank-metric code with minimum distance $m$, and let
$$S_A=\{(x,xA):x\in\mathbb F_q^m\}\subseteq\Fq^m\times\Fq^n, \quad \mathcal A_{\mathcal D}=\{S_A:A\in\mathcal D\}.$$
Since $\mathcal{D}$ has minimum distance $m$, then $\mathcal A_{\mathcal D}$ is a partial $m$-spread whose elements are disjoint from $S_\infty=\{0\}\times\mathbb F_q^n$.

\begin{proposition}[{\cite[Theorem 6.3]{GRSZ2024}}]
Let $\mathcal D\subseteq\Fq^{m\times n}$, with $m\le n$, be a rank-metric code with minimum distance $m$ and let $t$ be a positive integer. Then $\rho(\mathcal D)\ge m-t$ if and only if there exists an $(\mathcal A_{\mathcal D},t)$-evasive $m$-subspace meeting $S_\infty$ trivially.
\end{proposition}

The subspace $U_{f,W}$ of Remark \ref{rem:d:evasive} does not fit the framework of \cite[Theorem 6.3]{GRSZ2024} directly. Indeed, the family $\mathcal S=\{P_\lambda:\lambda\in\Fm\}$ is obtained from the Desarguesian spread, with $P$ as the missing spread element, where $U_{f,W}$ has dimension $n=d+\delta$ and satisfies $U_{f,W}\cap P=\{0\}\times W.$
In order to highlight this connection, we need to do the quotient over $\{0\}\times W$. Consider
$$\Psi_W\colon (x,y)\in V\times\Fm\longrightarrow (x,\pi_W(y))\in V\times(\Fm/W)$$
and, for $g\in\mathrm{Hom}_{\Fq}(V,\Fm/W)$, let $\Gamma_g=\{(x,g(x)):x\in V\}.$ 

Writing $\widetilde f=\pi_W\circ f$, we obtain
$\Psi_W(U_{f,W})=\Gamma_{\widetilde f}$ and $\Psi_W\bigl(P_\lambda\cap(V\times\Fm)\bigr)=\Gamma_{A_\lambda},$
where $A_\lambda\colon x\in V\longrightarrow\pi_W(\lambda x)\in \Fm/W.$
Moreover, $\ker(\Psi_W\cap U_{f,W})=\{0\}\times W$, so that the intersection with $P$ is trivial in the quotient. In particular, $\Gamma_{\widetilde f}$ meets $\{0\}\times(\Fm/W)$ trivially.

The relevant intersection dimensions are preserved in the following form:
$$
\begin{aligned}
\wt_{U_{f,W}}(P_\lambda)=\dim_{\Fq}(\{x\in V:f(x)-\lambda x\in W\})=\dim_{\Fq}(\ker(\widetilde f-A_\lambda))=d-\rk_{\Fq}(\widetilde f-A_\lambda).
\end{aligned}
$$
Thus, after choosing bases of $V$ and $\Fm/W$, the quotient construction yields exactly the rank-intersection mechanism underlying \cite[Theorem 6.3]{GRSZ2024}. More precisely, the family $\{\Gamma_{A_\lambda}:\lambda\in\Fm\}$ is the family of graphs associated with the auxiliary matrix code $\mathcal A_{V,W}$, while $\Gamma_{\widetilde f}$ is the graph associated with the ambient map $\widetilde f$.

Taking $t=n-\min\{m,n\}+1,$ and recalling that $d=n-\delta$, we have $d-t=\min\{m,n\}-\delta-1.$ Hence, the equivalence between conditions $(ii)$ and $(iii)$ in Proposition \ref{prop:caractd} is governed by the same rank-intersection identity appearing in \cite[Theorem 6.3]{GRSZ2024}. The difference is that, in our setting, this identity is available for every auxiliary code $\mathcal A_{V,W}$, independently of whether its graphs form a partial spread. Moreover, \cite[Theorem 6.3]{GRSZ2024} assumes that the matrix code has minimum distance equal to its number of rows. For $\mathcal A_{V,W}$, this corresponds, when no transposition is needed, to requiring every nonzero element of $\mathcal A_{V,W}$ to be injective. Equivalently, for every $\lambda\in\Fm$,
$$
\lambda V\cap W\neq\{0\}
\quad\Longrightarrow\quad
\lambda V\subseteq W.
$$
Under this condition, the distinct graphs $\Gamma_{A_\lambda}$ form a partial $d$-spread. In general this need not happen. For instance, if $0<\delta<d$, choose $v\in V\setminus\{0\}$ and $w\in W\setminus\{0\}$ and set $\lambda=w/v$. Then $\lambda V\cap W\neq\{0\}$, whereas $\lambda V\not\subseteq W$ because $\dim_{\Fq}(\lambda V)=d>\delta$.

Thus, after quotienting by $\{0\}\times W$, the characterization in Proposition \ref{prop:caractd} is governed by the same rank-intersection mechanism as \cite[Theorem 6.3]{GRSZ2024}, but without requiring the resulting family of graphs to form a partial spread. The additional feature of Theorem \ref{thm:assmat} is that the covering radius of the original code is recovered by optimizing over the auxiliary codes $\mathcal A_{V,W}$. The connection with clubs provides a further geometric ingredient which is not involved in \cite[Theorem 6.3]{GRSZ2024}

\section*{Declaration on the use of AI tools}
During the preparation of this work, the authors used generative AI tools (GPT-5.6 Sol) for two purposes only: to improve
the language and readability of the manuscript, and to assist in writing the MAGMA code used for the computations in Section~\ref{subsec:special_cases}.
All mathematical content, including the definitions, results and proofs, was developed by the authors. The authors reviewed and edited all AI-assisted text, checked the code and verified its output, and take full responsibility for the content of this publication.

\section*{Acknowledgments}
This research has been partially supported by the Italian National Group for Algebraic and Geometric Structures and their Applications (GNSAGA -- INdAM). A. Neri is supported by the University of Naples Federico II through the FRA2024 Project CUP E63C26001530001  ``NHACT -- New Horizons in Algebraic Coding Theory''.

\bibliographystyle{abbrvurl}
\bibliography{references.bib}

\end{document}